\documentclass[10pt]{article}
\usepackage[left=1.5in,right=1.5in]{geometry}
\pdfoutput=1
\usepackage{graphicx}
\usepackage{mathptmx}
\usepackage{mathrsfs}
\usepackage{amssymb}
\usepackage{amsmath}
\usepackage{amsmath,amsfonts}
\usepackage{algorithm}
\usepackage{algorithmic}
\usepackage[numbers,square]{natbib}
\usepackage{amscd}
\usepackage[mathscr]{eucal}
\usepackage{tikz}
\usepackage[doublespacing]{setspace}
\usepackage{fancyhdr}
\usepackage{lineno}

\usepackage{enumitem}
\usepackage{stackrel}

\numberwithin{equation}{section}

\newtheorem{theorem}{Theorem}[section]
\newtheorem{thm}{Theorem}[section]
\numberwithin{theorem}{section}
\newtheorem{rem}[theorem]{Remark}
\newtheorem{defn}[theorem]{Definition}

\newenvironment{proof}{\noindent\\ \noindent\relax{\sc
     Proof}}{{\samepage\par\nopagebreak\hbox
     to\hsize{\hfill$\Box$}}}
\newcommand{\be}{\begin{equation}} \newcommand{\ee}{\end{equation}}
\newcommand{\bd}{\begin{displaymath}} \newcommand{\ed}{\end{displaymath}}
\newcommand{\ben}{\begin{enumerate}} \newcommand{\een}{\end{enumerate}}
\newcommand{\bi}{\begin{itemize}} \newcommand{\ei}{\end{itemize}}

\newcommand{\ud}{\mathrm{d}}

\newcommand{\Expectation}[1]{\operatorname{E}\left[ #1 \right]}

\newcommand{\V}[1]{\mathbb{V}\left( #1 \right)}
\newcommand{\C}[1]{\mathbb{C}\left( #1 \right)}
\newcommand{\Et}[1]{\Expectation{ #1 \vert \mathcal{Y}_{n}}}
\newcommand{\Vart}[1]{\V{ #1 \vert  \mathcal{Y}_{n}}}
\newcommand{\covx}[2]{\C{ #1,#2}}

\newcounter{rEPP}

\newcounter{rEexpPsi}

\newcounter{rEtauP}

\begin{document}


\title{Normal approximation for mixtures of normal distributions and the evolution of phenotypic traits}

\date{}


\author{{\sc Krzysztof Bartoszek}\thanks{{krzysztof.bartoszek@liu.se, krzbar@protonmail.ch, 
Department of Computer and Information Science, Link\"oping University, 581 83 Link\"oping, Sweden}}
~and {\sc Torkel Erhardsson}
\thanks{{
torkel.erhardsson@liu.se, 
Department of Mathematics, 
Link\"oping University, 581 83 Link\"oping, Sweden}}}

\maketitle

\begin{abstract}
Explicit bounds are given for the Kolmogorov and Wasserstein distances between a mixture of normal distributions, by which we mean that the conditional distribution given some $\sigma$-algebra is normal, and a normal 
distribution with properly chosen parameter values. The bounds depend only on the first two moments of the first two conditional moments given the $\sigma$-algebra. The proof is based on Stein's method. As an application, we consider the Yule-Ornstein-Uhlenbeck model, used in the field
of phylogenetic comparative methods. We obtain bounds in either distance between the distribution of the average value of a phenotypic trait 
over \emph{n} related species, and a normal distribution. The bounds imply and extend earlier limit theorems by Bartoszek and Sagitov.
\end{abstract}


\section{Introduction}\label{sec:introduction}
In this paper we derive upper bounds for the Kolmogorov and Wasserstein distances between a mixture of normal distributions and a normal distribution with properly chosen parameter values. Here, a random variable $X$ is said to have a mixture of normal distributions if there exists a $\sigma$-algebra $\mathscr{G}$ such that the conditional distribution of $X$ given $\mathscr{G}$ is normal. Also, for comparison and completeness, lower bounds for both distances are derived.

To see why this is of interest, suppose that a random sequence $\{X_n;n=0,1,\ldots\}$ converges in distribution to a normal random variable $Z$. If $\mathscr{L}(Z)$ is used instead of $\mathscr{L}(X_n)$ for the (approximate) computation of the expectation $\mathbb{E}(h(X_n))$, where $h:\mathbb{R}\to\mathbb{R}$ is a measurable function, an approximation error $\mathbb{E}(h(X_n))-\mathbb{E}(h(Z))$ is incurred, about which the limit theorem \emph{per se} gives no information. In order to control this error, it is natural to use a metric on the space of probability measures on $(\mathbb{R},\mathscr{R})$, and try to bound the distance between $\mathscr{L}(X_n)$ and $\mathscr{L}(Z)$. A common choice is the Kolmogorov distance, which is defined for any two random variables $X$ and $Z$ with probability distributions $\mu_1$ and $\mu_2$ by
\[d_K\bigl(\mu_1,\mu_2\bigr) = \sup_{x\in\mathbb{R}}\bigl|\mathbb{P}(X\leq x)-\mathbb{P}(Z\leq x)\bigr|.\]
Another possibility is the Wasserstein distance, defined by
\[d_W\bigl(\mu_1,\mu_2\bigr) = \sup_{h\in\mathcal{H}_1}\bigl|\mathbb{E}(h(X))-\mathbb{E}(h(Z))\bigr|,\]
where $\mathcal{H}_1$ is the class of Lipschitz functions with Lipschitz constant bounded by 1.

In Section~\ref{sec:normalapproxnormalmixture}, we derive bounds in both distances between the probability distribution of a random variable $X$, which has a mixture of normal distributions, and a normally distributed random variable (Theorems~\ref{T:Kolmogorovboundmixednormal} and \ref{T:Wassersteinboundmixednormal}). The bounds depend only on the first two moments of the first two conditional moments given the ``mixing'' $\sigma$-algebra. The main tool used is Stein's method, a powerful technique introduced in Stein \cite{Stein1972}. At the core of this method is a functional equation called the Stein equation:
\[f'(x) - xf(x) = I_{(-\infty,z]}(x) - \Phi(z) \qquad\forall x\in\mathbb{R},\]
where $\Phi$ is the cumulative distribution function of the $\textnormal{N}(0,1)$ distribution. By taking expectations with respect to $\mathscr{L}(X)$ on both sides, and using analytical properties of the solution function $f$, bounds can be obtained for the Kolmogorov distance between $\mathscr{L}(X)$ and $\textnormal{N}(0,1)$. While this is easiest if $X$ is a sum of locally dependent random variables, the use of couplings and other special devices has made it possible to handle many other situations. There are also extensions of the method which allow for other approximating distributions to be used, such as Poisson and compound Poisson distributions and multivariate normal distributions. Since its introduction, the number of applications of the method has grown very large. For more details and many examples, see Barbour and Chen~\cite{BarbourChen2005a}, \cite{BarbourChen2005b}, and the references therein. 

In the second part of the paper we apply the obtained results to branching Ornstein-Uhlenbeck processes. A one-dimensional Ornstein-Uhlenbeck (OU)
process is a stochastic process that follows a linear stochastic differential equation of the form
\be\label{E:OUsde}
\ud X(t) = -\alpha X(t) \ud t + \sigma_{a} \ud W(t) \qquad\forall t\geq 0,
\ee
where $\alpha,\sigma_a >0$, and $\{W(t);t\geq 0\}$ is a standard Wiener process. In the subfield of evolutionary biology called phylogenetic comparative methods, processes like \eqref{E:OUsde} 
are used for modelling the evolution of
phenotypic traits, such as body size, at the between-species level, in the following way: an Ornstein-Uhlenbeck process evolves on top of a possibly random phylogenetic tree, 
by which we mean a (random) 
directed acyclic graph with weights on edges that correspond to edge length, and nodes corresponding to the branching 
events in the tree, see Fig.~\ref{F:TreeOU}. In the Yule--Ornstein-Uhlenbeck (YOU) model, which we consider here, each speciation (=branching) point is binary, and the edge lengths 
are independent exponentially distributed random variables. This so-called pure birth tree is stopped just before 
the $n$th speciation event,
i.e., it has $n$ leaves (= tips). Without loss of generality we fix the birth rate to 1. 
Varying the birth rate will only have the effect of rescaling time and will not add anything substantial to our results.

\begin{figure}
\begin{center}
\includegraphics[width=0.45\textwidth]{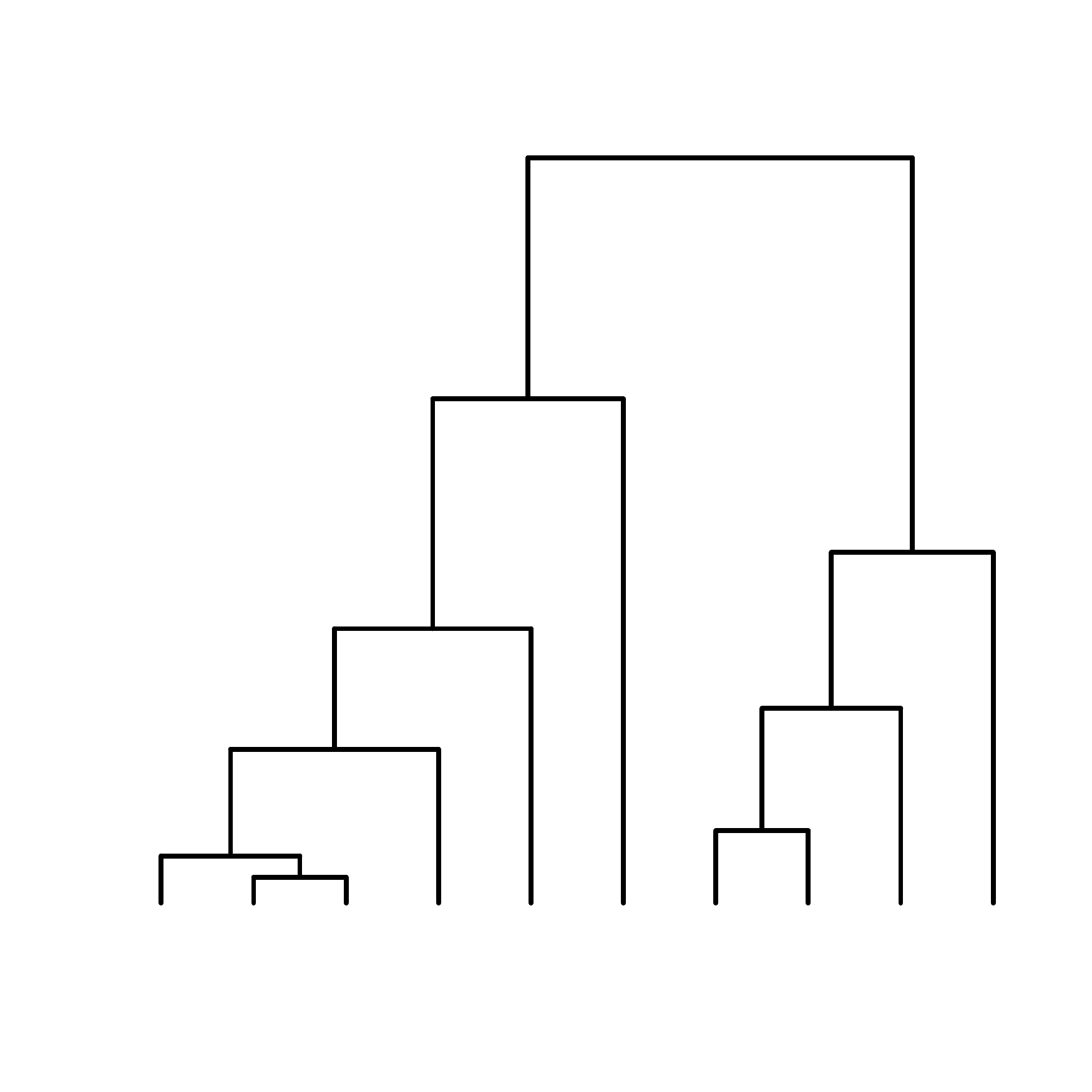}
\includegraphics[width=0.45\textwidth]{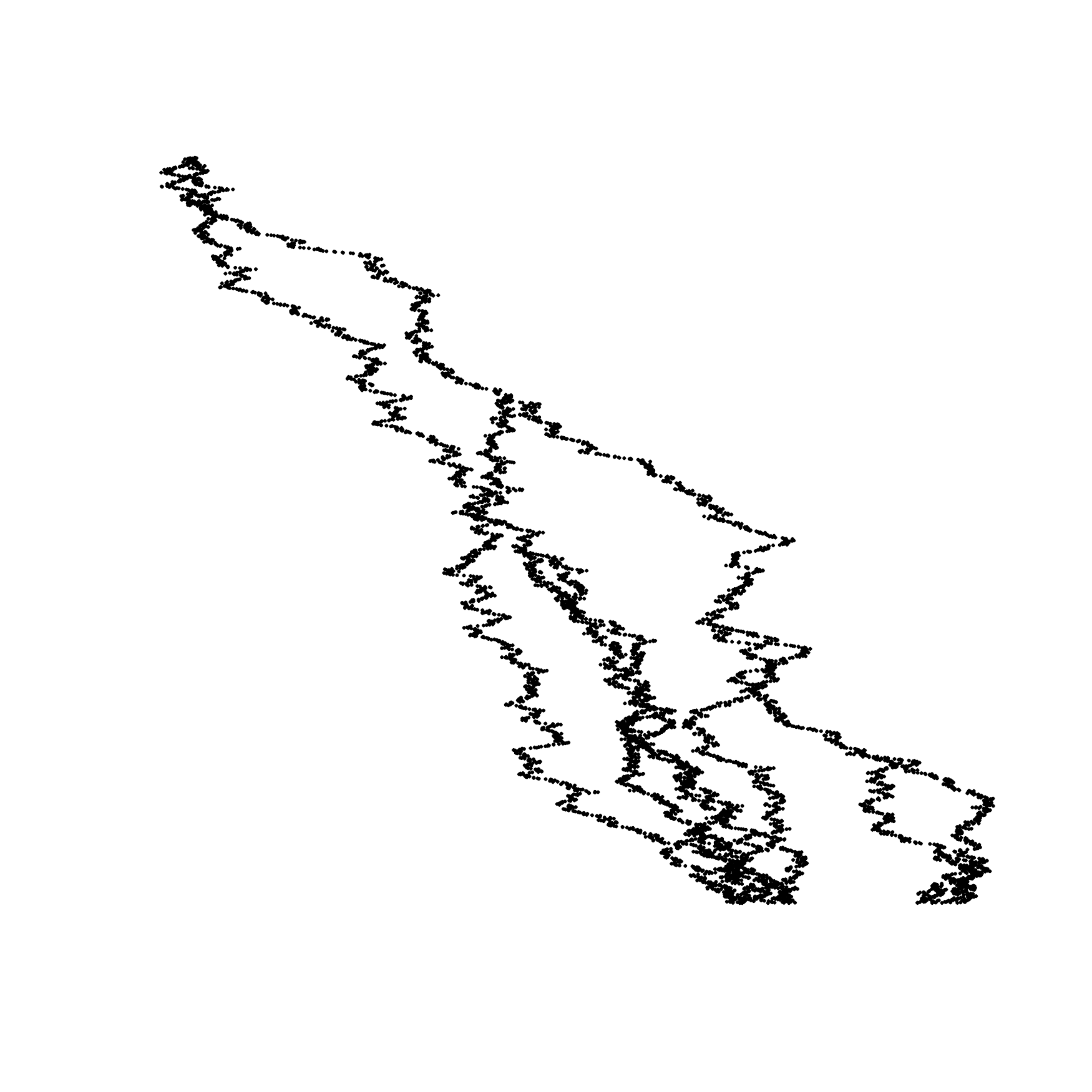}
\caption{Left: an example phylogenetic tree with $10$ leaves, simulated using the R \cite{R2017} package TreeSim \cite{Stadler2011}. 
Right: an OU process with parameters $\alpha=1,~\sigma_{a}=1/2,~X(0)=-3$ evolving on top of this tree, 
simulated using the R package mvSLOUCH \cite{BartoszekPienaarMostadAnderssonHansen2012}.}\label{F:TreeOU}
\end{center}
\end{figure}

In the YUO model, along each edge (= branch) the process describing the phenotypic trait behaves as defined by \eqref{E:OUsde}. Then, at a speciation
point the process splits into as many copies as there are descendant branches. At the start of each descendant branch the process starts with the value at
which the ancestral branch ended (the starting value is the same for all descendant branches). From that point onward, on each descendant lineage the processes
behave independently.

The YOU model can be further extended by allowing for jumps; see Bokma \cite{Bokma2002}.
A particular type of jumps that can serve as a starting point for mathematical analysis,
is when a jump takes place just after a speciation event, independently on each descendant lineage, with a probability $p$ that may be dependent on the speciation event; see Section~\ref{sec:applicationtoyoujumpsmodel} for more details.

In the context of evolutionary biology, the observed phenotypic data are the values of the process at the tips, 
$\{X_{i}\}_{i=1}^{n}$. Of particular interest are central
limit theorems for the sample average, $\overline{X}_n$, or more generally for functionals
of the observed data (see e.g.~Ren et al.~\cite{RenSongZhang2014}, Adamczak and Mi{\l }o\'s \cite{AdamczakMilos2015}, Bartoszek and Sagitov \cite{BartoszekSagitov2015}, An\'e et al.~\cite{AneHoRoch2017}, Bartoszek \cite{Bartoszek2020}, and a multitude of other works). If the drift of the OU process is fast enough, then 
one can show convergence in distribution for $\overline{X}_n$ to a normal limit. However, if the drift is slow, then the dependencies induced by common ancestry persist and statements about the limit are more 
involved. The above was shown for the YOU model in \cite{BartoszekSagitov2015}, while the YOU model with normally distributed jumps was considered in \cite{Bartoszek2020}. In the slow drift regime one can show
$L^{2}$ convergence (see e.g.~\cite{AdamczakMilos2015}, \cite{Bartoszek2020}, \cite{BartoszekSagitov2015}). However, so far there is no complete characteristic of the limit in this case.

In Sections~\ref{sec:applicationtoyoumodel} and~\ref{sec:applicationtoyoujumpsmodel} of the present paper, we extend the central limit theorems for $\overline{X}_n$ by giving bounds for the Kolmogorov and Wasserstein distances between the distribution of $\overline{X}_n$ and properly chosen normal distributions (Theorems~\ref{T:YOU}, \ref{T:YOUWasserstein}, \ref{T:YOUj} and \ref{T:YOUjWasserstein}), which converge weakly to the limiting normal distributions of \cite{BartoszekSagitov2015} and \cite{Bartoszek2020} as $n\to\infty$. The key observation is that conditional on the tree (and the locations of jumps), $\overline{X}_n$ is a linear combination of normally distributed random variables, which makes it possible to apply Theorems~\ref{T:Kolmogorovboundmixednormal} and \ref{T:Wassersteinboundmixednormal}. One needs to compute the first two moments of the conditional expectation and variance of $\overline{X}_n$, which requires a careful analysis of the random quantities involved, e.g., the heights in the tree and speciation
events along lineages, but a considerable part of this work was done in \cite{BartoszekSagitov2015} and \cite{Bartoszek2020} and can be re-used here.

Lastly, in the Appendix, for the sake of comparison and completeness, we state and prove lower bounds in either distance between the probability distributions of a random variable $X$, which has a mixture of normal distributions, and a normally distributed random variable. The proof is based on ideas in Barbour and Hall \cite{BarbourHall1984}. 

\section{Normal approximation for mixtures of normal distributions}\label{sec:normalapproxnormalmixture}
A metric $d(\cdot,\cdot)$ on the space of probability measures on a measurable space $(\Omega,\mathscr{F})$ is called an integral probability metric, see M\"uller~\cite{Muller1997}, if
\begin{equation} \label{E:integralprobabilitymetricdefinition}
d(\mu_1,\mu_2) = \sup_{h\in\mathcal{H}}\bigl|\int h(x)d\mu_1(x) - \int h(x)d\mu_2(x)\bigr|,
\end{equation}
where $\mathcal{H}$ is a class of measurable functions $h:\Omega\to\mathbb{R}$ called the generating class. Our interest is in two integral probability metrics on the space of probability measures on $(\mathbb{R},\mathscr{R})$: the Kolmogorov distance $d_K$, for which $\mathcal{H}$ is the set of indicator functions of half-lines, $\mathcal{H}_0 = \{I_{(-\infty,z]}(\cdot); z\in\mathbb{R}\}$, and the Wasserstein distance $d_W$, for which $\mathcal{H}$ is the set $\mathcal{H}_1$ of Lipschitz functions with Lipschitz constant bounded by 1. It is well-known that for sequences of probability measures on $(\mathbb{R},\mathscr{R})$, convergence in either distance implies the usual weak convergence; see Section 4 in \cite{Muller1997}.

Also, the Kolmogorov distance is scale (and location) invariant, in the sense that 
\begin{equation} \label{E:Kolmogorovscaleinvariance}
d_K\Bigl(\mathscr{L}(X),\mathscr{L}(Y)\Bigr) = d_K\Bigl(\mathscr{L}(\frac{X-\mu}{\sigma}),\mathscr{L}(\frac{Y-\mu}{\sigma})\Bigr) \qquad\forall \mu\in\mathbb{R},\sigma>0,
\end{equation}
for any pair of random variables $X$ and $Y$. This follows from \eqref{E:integralprobabilitymetricdefinition} and the fact that
\[\mathcal{H}_0 = \{I_{(-\infty,\sigma z+\mu]}(\cdot); z\in\mathbb{R}\} \qquad\forall \mu\in\mathbb{R},\sigma>0.\]
The Wasserstein distance is not scale invariant, but has the property
\begin{equation} \label{E:Wassersteinpseudoscaleinvariance}
d_W\Bigl(\mathscr{L}(X),\mathscr{L}(Y)\Bigr) = \sigma d_W\Bigl(\mathscr{L}(\frac{X-\mu}{\sigma}),\mathscr{L}(\frac{Y-\mu}{\sigma})\Bigr) \qquad\forall \mu\in\mathbb{R},\sigma>0,
\end{equation}
which follows from \eqref{E:integralprobabilitymetricdefinition} and the fact that for each $\mu\in\mathbb{R}$, $\sigma>0$, the mapping $\xi:\mathcal{H}_1\to\mathcal{H}_1$, defined by: $\xi h(x) = \sigma h(\frac{x - \mu}{\sigma})$, is a bijection.

Our main results are contained in Theorem~\ref{T:Kolmogorovboundmixednormal} (Kolmogorov distance) and Theorem~\ref{T:Wassersteinboundmixednormal} (Wasserstein distance). 

\begin{thm}\label{T:Kolmogorovboundmixednormal}
Let $X$ be a real valued random variable such that $\mathbb{E}(X^2)<\infty$, and let $\mathscr{G}$ be a $\sigma$-algebra such that the regular conditional distribution of $X$ given $\mathscr{G}$ is normal. Then,
\[d_K\Bigl(\mathscr{L}\bigl(\frac{X-\mathbb{E}(X)}{\sqrt{\mathbb{E}(\mathbb{V}(X|\mathscr{G}))}}\bigr),\textnormal{N}(0,1))\Bigr) = d_K\Bigl(\mathscr{L}(X),\textnormal{N}\bigl(\mathbb{E}(X),\mathbb{E}(\mathbb{V}(X|\mathscr{G}))\bigr)\Bigr)\]
\[\leq \frac{\sqrt{\mathbb{V}\bigl(\mathbb{V}(X|\mathscr{G})\bigr)}}{\mathbb{E}\bigl(\mathbb{V}(X|\mathscr{G})\bigr)} + \frac{\mathbb{V}\bigl(\mathbb{E}(X|\mathscr{G})\bigr)}{\mathbb{E}\bigl(\mathbb{V}(X|\mathscr{G})\bigr)} + \sqrt{\frac{2}{\pi}}\frac{\sqrt{\mathbb{V}\bigl(\mathbb{E}(X|\mathscr{G})\bigr)}\mathbb{V}\bigl(\mathbb{V}(X|\mathscr{G})\bigr)^{1/4}}{\mathbb{E}\bigl(\mathbb{V}(X|\mathscr{G})\bigr)}.\]
\end{thm}

\begin{proof}
The following identity, called the Stein identity for the N$(0,1)$ distribution, was originally derived in \cite{Stein1972} (for more information, see Chen and Shao \cite{ChenShao2005} and the references therein): if $Z$ is any real valued random variable, then $Z\sim\textnormal{N}(0,1)$ if and only if
\begin{equation} \label{E:normalsteinidentity}
\mathbb{E}\bigl(f'(Z) - Zf(Z)\bigr) = 0 \qquad\forall f\in\mathcal{C}_{bd},
\end{equation}
where $\mathcal{C}_{bd}$ is the set of continuous, piecewise continuously differentiable functions $f:\mathbb{R}\to\mathbb{R}$ such that $\mathbb{E}(|f'(Z_{0,1})|)<\infty$ if $Z_{0,1}\sim\textnormal{N}(0,1)$.

Using \eqref{E:normalsteinidentity}, we shall first derive a similar Stein identity for the $\textnormal{N}(\mu,\sigma^2)$ distribution, where $\mu\in\mathbb{R}$ and $\sigma\in(0,\infty)$: if $W$ is any real valued random variable, then $W\sim\textnormal{N}(\mu,\sigma^2)$ if and only if
\begin{equation} \label{E:generalnormalsteinidentity2}
\mathbb{E}\bigl(\sigma^2g'(W) - (W-\mu)g(W)\bigr) = 0 \qquad\forall g\in\mathcal{C}_{bd}^{\mu,\sigma},
\end{equation}
where $\mathcal{C}_{bd}^{\mu,\sigma}$ is the set of continuous, piecewise continuously differentiable functions $g:\mathbb{R}\to\mathbb{R}$ such that $\mathbb{E}(|g'(Z_{\mu,\sigma})|)<\infty$ if $Z_{\mu,\sigma}\sim\textnormal{N}(\mu,\sigma^2)$. To prove \eqref{E:generalnormalsteinidentity2}, we define the random variable $Z$ by $Z=\frac{1}{\sigma}(W-\mu)$, and note that $Z\sim\textnormal{N}(0,1)$ if and only if $W\sim\textnormal{N}(\mu,\sigma^2)$. We also define the mapping $T:\mathcal{C}_{bd}^{\mu,\sigma}\to\mathcal{C}_{bd}$ by $Tg(x) = \sigma g(\sigma x + \mu)$. $T$ is easily seen to be a bijection with inverse $T^{-1}f(y) = \frac{1}{\sigma}f(\frac{y-\mu}{\sigma})$. This gives:
\[\sigma^2g'(W) - (W-\mu)g(W) = \sigma^2g'(\sigma Z+\mu) - (\sigma Z+\mu-\mu)g(\sigma Z+\mu)\]
\[= [Tg]'(Z) - Z [Tg](Z) \qquad\forall g\in\mathcal{C}_{bd}^{\mu,\sigma},\]
and this in combination with \eqref{E:normalsteinidentity} gives \eqref{E:generalnormalsteinidentity2}.

We next consider the following functional equation, which we propose to call the Stein equation for the $\textnormal{N}(\mu,\sigma^2)$ distribution. It arises in a natural way from \eqref{E:generalnormalsteinidentity2}:
\begin{equation} \label{E:generalnormalsteineq2}
\sigma^2g'(y) - (y-\mu)g(y) = I_{(-\infty,z]}\bigl(\frac{y-\mu}{\sigma}\bigr) - \Phi(z) \qquad\forall y\in\mathbb{R},
\end{equation}
where $z\in\mathbb{R}$. For each fixed $z\in\mathbb{R}$, it is clear that a function $g\in\mathcal{C}_{bd}^{\mu,\sigma}$ satisfies \eqref{E:generalnormalsteineq2} if and only if the function $f = Tg\in\mathcal{C}_{bd}$ (defined above) satisfies the functional equation
\begin{equation} \label{E:equivalentnormalsteineq2}
f'(x) - xf(x) = I_{(-\infty,z]}(x) - \Phi(z) \qquad\forall x\in\mathbb{R},
\end{equation}
which is the classical Stein equation for the N$(0,1)$ distribution. We obtain from Section 2.1 in \cite{ChenShao2005} that \eqref{E:equivalentnormalsteineq} has the solution $f = f_z$, where
\[f_z(x) = e^{x^2/2}\int_{-\infty}^x\bigl[I_{(-\infty,z]}(u) - \Phi(z)\bigr]e^{-u^2/2}du\qquad\forall x\in\mathbb{R}.\]
It is also shown in Section 2.2 in \cite{ChenShao2005} that $f_z$ is bounded, continuous, and continuously differentiable except at $x=z$. Moreover, $f_z$ satisfies:
\[0<f_z(x)\leq \frac{\sqrt{2\pi}}{4} \quad\forall x\in\mathbb{R}; \qquad |f_z'(x)|\leq 1 \quad\forall x\in\mathbb{R}.\]
Therefore, the function $g_z = T^{-1}f_z$, explicitly given by $g_z(y)= \frac{1}{\sigma}f_z(\frac{y-\mu}{\sigma})$, is a solution to \eqref{E:generalnormalsteineq2}. $g_z$ is bounded, continuous, and continuously differentiable except at $y=\sigma z+\mu$, and satisfies:
\begin{equation} \label{E:generalsteineqsolutionbounds2}
0<g_z(y)\leq \frac{\sqrt{2\pi}}{4\sigma} \quad\forall y\in\mathbb{R}; \qquad |g_z'(y)| \leq \frac{1}{\sigma^2} \quad\forall y\in\mathbb{R}.
\end{equation}

For the remainder of the proof, we define for convenience $\mathcal{C}_{bbd}$ as the set of bounded, continuous, piecewise continuously differentiable functions $g:\mathbb{R}\to\mathbb{R}$ with bounded derivative. By definition, $\mathcal{C}_{bbd}\subset\mathcal{C}_{bd}^{\mu,\sigma}$ for each $\mu\in\mathbb{R}$, $\sigma\in(0,\infty)$, and by \eqref{E:generalsteineqsolutionbounds2}, $g_z\in\mathcal{C}_{bbd}$ for each $z\in\mathbb{R}$. Recalling that the random variable $X$ has a conditionally normal distribution given $\mathscr{G}$, we obtain from \eqref{E:generalnormalsteinidentity2}:
\[\mathbb{E}\Bigl(\mathbb{V}(X|\mathscr{G}) g'(X) - \bigl(X-\mathbb{E}(X|\mathscr{G})\bigr)g(X)\bigr|\mathscr{G}\Bigr) = 0 \qquad\textnormal{$P$-a.s.}\qquad\forall g\in\mathcal{C}_{bbd}.\]
Taking expectations and rewriting, this gives:
\begin{equation} \label{E:expectedconditionalsteinidentity2}
\mathbb{E}\bigl(\mathbb{V}(X|\mathscr{G}) g'(X) + \mathbb{E}(X|\mathscr{G})g(X)\bigr) = \mathbb{E}\bigl(Xg(X)\bigr) \qquad\forall g\in\mathcal{C}_{bbd}.
\end{equation}
From the definition of Kolmogorov distance and \eqref{E:Kolmogorovscaleinvariance}, it follows that for any $\mu\in\mathbb{R}$ and $\sigma\in(0,\infty)$,
\[d_K\bigl(\mathscr{L}(X),\textnormal{N$(\mu,\sigma^2)$}\bigr) = d_K\bigl(\mathscr{L}\bigl(\frac{X-\mu}{\sigma}\bigr),\textnormal{N$(0,1)$}\bigr) = \sup_{z\in\mathbb{R}}\bigl|\mathbb{P}\bigl(\frac{X-\mu}{\sigma} \leq z\bigr) - \Phi(z)\bigr|,\]
and, using \eqref{E:generalnormalsteineq2} and \eqref{E:expectedconditionalsteinidentity2},
\begin{equation} \label{E:expectedconditionalsteinrawbound2}
\begin{split}
\mathbb{P}\bigl(\frac{X-\mu}{\sigma} \leq z\bigr) - \Phi(z) &= \mathbb{E}\bigl(\sigma^2 g_z'(X) - (X-\mu) g_z(X)\bigr)\\
= \mathbb{E}\bigl((\sigma^2-\mathbb{V}(X|\mathscr{G})) g_z'(X) &+ (\mu-\mathbb{E}(X|\mathscr{G})) g_z(X)\bigr) \qquad\forall z\in\mathbb{R}.
\end{split}
\end{equation}
If we choose $\mu =\mathbb{E}(X)$ and $\sigma^2 = \mathbb{E}\bigl(\mathbb{V}(X|\mathscr{G})\bigr)$, we get:
\[\bigl|\mathbb{E}\bigl((\sigma^2-\mathbb{V}(X|\mathscr{G})) g_z'(X)\bigr)\bigr| \leq \mathbb{E}\bigl(\bigl|\sigma^2-\mathbb{V}(X|\mathscr{G})\bigr|\bigr)\frac{1}{\sigma^2} \leq  \frac{\sqrt{\mathbb{V}\bigl(\mathbb{V}(X|\mathscr{G})\bigr)}}{\mathbb{E}\bigl(\mathbb{V}(X|\mathscr{G})\bigr)} \qquad\forall z\in\mathbb{R},\]
using \eqref{E:generalsteineqsolutionbounds2} and H\"older's inequality. For the second term on the right hand side of \eqref{E:expectedconditionalsteinrawbound2}, we will use a coupling, similar to the one used in the proof of Theorem 1.C in Barbour et al.~\cite{BarbourHolstJanson1992}; the latter theorem deals with Poisson approximations for mixtures of Poisson distributions. First, letting the random variable $Y\sim\textnormal{N}(\mu,\sigma^2)$ be independent of $\mathscr{G}$, we can write:
\[\mathbb{E}\bigl((\mu-\mathbb{E}(X|\mathscr{G})) g_z(X)\bigr) = \mathbb{E}\bigl((\mu-\mathbb{E}(X|\mathscr{G}))(g_z(X) - g_z(Y))\bigr)\]
\[= \mathbb{E}\bigl((\mu-\mathbb{E}(X|\mathscr{G}))I_A\mathbb{E}(g_z(X) - g_z(Y)|\mathscr{G})\bigr)\]
\[+ \mathbb{E}\bigl((\mu-\mathbb{E}(X|\mathscr{G}))I_{A^c}\mathbb{E}(g_z(X) - g_z(Y)|\mathscr{G})\bigr) \qquad\forall z\in\mathbb{R},\]
where $A=\{\sigma^2 \leq\mathbb{V}(X|\mathscr{G})\}$. For each $\omega\in A$, we construct a probability space with two independent random variables $Y_1\sim\textnormal{N}(0,\sigma^2)$ and $Y_2\sim\textnormal{N}(0,\mathbb{V}(X|\mathscr{G})-\sigma^2)$, so that $\mathbb{E}(X|\mathscr{G}) + Y_1+Y_2\sim\textnormal{N}(\mathbb{E}(X|\mathscr{G}),\mathbb{V}(X|\mathscr{G}))$, and $\mu + Y_1\sim\textnormal{N}(\mu,\sigma^2)$. Using this coupling, and the fact that $\lVert g'_z\rVert = \sup_{x\in\mathbb{R}}|g'_z(x)| \leq \frac{1}{\sigma^2}$, we obtain:
\[\bigl|\mathbb{E}\bigl((\mu-\mathbb{E}(X|\mathscr{G}))I_A\mathbb{E}(g_z(X) - g_z(Y)|\mathscr{G})\bigr)\bigr|\]
\[= \bigl|\mathbb{E}\bigl((\mu-\mathbb{E}(X|\mathscr{G}))I_A\mathbb{E}(g_z(\mathbb{E}(X|\mathscr{G}) + Y_1+Y_2) - g_z(\mu + Y_1)|\mathscr{G})\bigr)\bigr|\]
\[\leq \mathbb{E}\bigl(\bigl|\mu-\mathbb{E}(X|\mathscr{G})\bigr| I_A\lVert g_z'\rVert \mathbb{E}(|\mu-\mathbb{E}(X|\mathscr{G}) - Y_2||\mathscr{G})\bigr)\]
\[\leq \frac{1}{\sigma^2}\mathbb{E}\bigl((\mu-\mathbb{E}(X|\mathscr{G}))^2I_A\bigr) + \frac{1}{\sigma^2}\mathbb{E}\bigl(\bigl|\mu-\mathbb{E}(X|\mathscr{G})\bigr|I_A\mathbb{E}(|Y_2||\mathscr{G})\bigr)\]
\[= \frac{1}{\sigma^2}\mathbb{E}\bigl((\mu-\mathbb{E}(X|\mathscr{G}))^2I_A\bigr) + \frac{1}{\sigma^2}\mathbb{E}\bigl(\bigl|\mu-\mathbb{E}(X|\mathscr{G})\bigr|I_A\sqrt{\frac{2}{\pi}\bigl|\sigma^2-\mathbb{V}(X|\mathscr{G})\bigr|}\bigr)\qquad\forall z\in\mathbb{R}.\]
Similarly, for each $\omega\in A^c$, we construct a probability space with two independent random variables $\widehat Y_1\sim\textnormal{N}(0,\mathbb{V}(X|\mathscr{G}))$ and $\widehat Y_2\sim\textnormal{N}(0,\sigma^2-\mathbb{V}(X|\mathscr{G}))$, so that $\mathbb{E}(X|\mathscr{G}) + \widehat Y_1\sim\textnormal{N}(\mathbb{E}(X|\mathscr{G}),\mathbb{V}(X|\mathscr{G}))$, and $\mu + \widehat Y_1 + \widehat Y_2\sim\textnormal{N}(\mu,\sigma^2)$. This gives, after some calculations,
\[\bigl|\mathbb{E}\bigl((\mu-\mathbb{E}(X|\mathscr{G}))I_{A^c}\mathbb{E}(g_z(X) - g_z(Y)|\mathscr{G})\bigr)\bigr|\]
\[\leq \frac{1}{\sigma^2}\mathbb{E}\bigl((\mu-\mathbb{E}(X|\mathscr{G}))^2I_{A^c}\bigr) + \frac{1}{\sigma^2}\mathbb{E}\bigl(\bigl|\mu-\mathbb{E}(X|\mathscr{G})\bigr|I_{A^c}\sqrt{\frac{2}{\pi}\bigl|\sigma^2-\mathbb{V}(X|\mathscr{G})\bigr|}\bigr)\qquad\forall z\in\mathbb{R}.\]
Combining these two bounds, we get for the second term on the right hand side of \eqref{E:expectedconditionalsteinrawbound2}:
\[\bigl|\mathbb{E}\bigl((\mu-\mathbb{E}(X|\mathscr{G})) g_z(X)\bigr)\bigr|\]
\[\leq \frac{1}{\sigma^2}\mathbb{E}\bigl((\mu-\mathbb{E}(X|\mathscr{G}))^2\bigr) + \frac{1}{\sigma^2}\mathbb{E}\bigl(\bigl|\mu-\mathbb{E}(X|\mathscr{G})\bigr|\sqrt{\frac{2}{\pi}\bigl|\sigma^2-\mathbb{V}(X|\mathscr{G})\bigr|}\bigr)\]
\[\leq \frac{1}{\sigma^2}\mathbb{E}\bigl((\mu-\mathbb{E}(X|\mathscr{G}))^2\bigr) + \frac{1}{\sigma^2}\sqrt{\mathbb{E}\bigl((\mu-\mathbb{E}(X|\mathscr{G}))^2\bigr)}\sqrt{\frac{2}{\pi}\mathbb{E}(\bigl|\sigma^2-\mathbb{V}(X|\mathscr{G})\bigr|)}\]
\[\leq \frac{\mathbb{V}\bigl(\mathbb{E}(X|\mathscr{G})\bigr)}{\mathbb{E}\bigl(\mathbb{V}(X|\mathscr{G})\bigr)} + \sqrt{\frac{2}{\pi}}\frac{\sqrt{\mathbb{V}\bigl(\mathbb{E}(X|\mathscr{G})\bigr)}\mathbb{V}\bigl(\mathbb{V}(X|\mathscr{G})\bigr)^{1/4}}{\mathbb{E}\bigl(\mathbb{V}(X|\mathscr{G})\bigr)}\qquad\forall z\in\mathbb{R}.\]
\end{proof}

\begin{rem} \label{R:Kolmogorovboundtwonormals}
In the case when $\mathbb{E}(X|\mathscr{G})\equiv m$ and $\mathbb{V}(X|\mathscr{G})\equiv \tau^2$ for deterministic constants $m\in\mathbb{R}$ and $\tau>0$, meaning that $X\sim\textnormal{N}(m,\tau^2)$ independently of $\mathscr{G}$, we obtain from \eqref{E:expectedconditionalsteinrawbound2} and \eqref{E:generalsteineqsolutionbounds2},
\[d_K\bigl(\textnormal{N$(m,\tau^2)$},\textnormal{N$(\mu,\sigma^2)$}\bigr)\leq \frac{1}{\sigma^2}|\sigma^2-\tau^2| + \frac{\sqrt{2\pi}}{4\sigma}|\mu-m|.\]
\end{rem}

Turning to Theorem~\ref{T:Wassersteinboundmixednormal}, we define $\mathcal{H}_2$ as the set of all real valued absolutely continuous functions on $(\mathbb{R},\mathscr{R})$, by which we mean all functions $h:\mathbb{R}\to\mathbb{R}$ such that $h$ has a derivative almost everywhere, $h'$ is Lebesgue integrable on every compact interval, and
\[h(b) - h(a) = \int_a^b h'(u)du \qquad\forall -\infty<a\leq b<\infty.\]
It is well-known that any Lipschitz continuous function $h:\mathbb{R}\to\mathbb{R}$ is absolutely continuous, and that $|h'(x)|\leq K$, where $K$ is the Lipschitz constant, for all $x\in\mathbb{R}$ where $h'(x)$ is defined. Moreover, as stated above, the Wasserstein distance on the space of probability measures on $(\mathbb{R},\mathscr{R})$ is defined by:
\[d_W\bigl(\mu_1,\mu_2\bigr) = \sup_{h\in\mathcal{H}_1}\bigl|\int h(x)d\mu_1(x) - \int h(x)d\mu_2(x)\bigr|,\]
where $\mathcal{H}_1$ is the set of all Lipschitz continuous functions with Lipschitz constant bounded by 1.

\begin{thm} \label{T:Wassersteinboundmixednormal}
Let $X$ be a real valued random variable such that $\mathbb{E}(X^2)<\infty$, and let $\mathscr{G}$ be a $\sigma$-algebra such that the regular conditional distribution of $X$ given $\mathscr{G}$ is normal. Then,
\[d_W\Bigl(\mathscr{L}\bigl(\frac{X-\mathbb{E}(X)}{\sqrt{\mathbb{E}(\mathbb{V}(X|\mathscr{G}))}}\bigr),\textnormal{N}(0,1))\Bigr) \leq \sqrt{\frac{2}{\pi}}\,\frac{\mathbb{V}\bigl(\mathbb{V}(X|\mathscr{G})\bigr)^{3/4}}{\mathbb{E}\bigl(\mathbb{V}(X|\mathscr{G})\bigr)^{3/2}} + \frac{\sqrt{\mathbb{V}\bigl(\mathbb{E}(X|\mathscr{G})\bigr)}\sqrt{\mathbb{V}\bigl(\mathbb{V}(X|\mathscr{G})\bigr)}}{\mathbb{E}\bigl(\mathbb{V}(X|\mathscr{G})\bigr)^{3/2}}\]
\[+ \frac{\mathbb{V}\bigl(\mathbb{E}(X|\mathscr{G})\bigr)}{\mathbb{E}\bigl(\mathbb{V}(X|\mathscr{G})\bigr)} + \sqrt{\frac{2}{\pi}}\frac{\sqrt{\mathbb{V}\bigl(\mathbb{E}(X|\mathscr{G})\bigr)}\mathbb{V}\bigl(\mathbb{V}(X|\mathscr{G})\bigr)^{1/4}}{\mathbb{E}\bigl(\mathbb{V}(X|\mathscr{G})\bigr)}.\]
\end{thm}

\begin{proof}
The first part of the proof is the same as for Theorem~\ref{T:Kolmogorovboundmixednormal}. However, as a Stein equation for the $\textnormal{N}(\mu,\sigma^2)$ distribution, we use, instead of \eqref{E:generalnormalsteineq2}:
\begin{equation} \label{E:generalnormalsteineq}
\sigma^2g'(y) - (y-\mu)g(y) = h(\frac{y-\mu}{\sigma}\bigr) - \mathbb{E}(h(Z_{0,1})) \qquad\forall y\in\mathbb{R},
\end{equation}
where $h\in\mathcal{H}_1$, and $Z_{0,1}\sim\textnormal{N}(0,1)$. For each $h\in\mathcal{H}_1$, it is clear that a function $g\in\mathcal{C}_{bd}^{\mu,\sigma}$ satisfies \eqref{E:generalnormalsteineq} if and only if the function $f = Tg\in\mathcal{C}_{bd}$ (defined in the proof of Theorem~\ref{T:Kolmogorovboundmixednormal}) satisfies the functional equation
\begin{equation} \label{E:equivalentnormalsteineq}
f'(x) - xf(x) = h(x) - \mathbb{E}(h(Z_{0,1})) \qquad\forall x\in\mathbb{R}.
\end{equation}
It is shown in \cite{ChenShao2005} that \eqref{E:equivalentnormalsteineq} has the solution $f = f_h$, where
\[f_h(x) = e^{x^2/2}\int_{-\infty}^x\bigl[h(u) - \mathbb{E}(h(Z_{0,1}))\bigr]e^{-u^2/2}du\qquad\forall x\in\mathbb{R}.\]
Moreover, for each $h\in\mathcal{H}_1$, $f_h$ is bounded, has an absolutely continuous derivative, and satisfies:
\[\lVert f_h\rVert\leq \min\bigl(\sqrt{\frac{\pi}{2}}\lVert h-\mathbb{E}(h(Z))\rVert,2\lVert h'\rVert\bigr);\]
\[\quad \lVert f'_h\rVert\leq \min\bigl(2\lVert h-\mathbb{E}(h(Z))\rVert,4\lVert h'\rVert\bigr); \qquad\lVert f''_h\rVert\leq 2\lVert h'\rVert,\]
where $\lVert\cdot\rVert$ denotes the (essential) supremum. Therefore, the function $g_h = T^{-1}f_h$, explicitly given by $g_h(y)= \frac{1}{\sigma}f_h(\frac{y-\mu}{\sigma})$, is a solution to \eqref{E:generalnormalsteineq} which is bounded, has an absolutely continuous derivative, and satisfies:
\begin{equation} \label{E:generalsteineqsolutionbounds}
\begin{split}
&\lVert g_h\rVert\leq \frac{1}{\sigma}\min\bigl(\sqrt{\frac{\pi}{2}}\lVert h-\mathbb{E}(h(Z))\rVert,2\lVert h'\rVert\bigr);\\
\lVert g'_h\rVert\leq \frac{1}{\sigma^2}&\min\bigl(2\lVert h-\mathbb{E}(h(Z))\rVert,4\lVert h'\rVert\bigr); \qquad \lVert g''_h\rVert\leq \frac{2}{\sigma^3}\lVert h'\rVert.
\end{split}
\end{equation}
As in the proof of Theorem~\ref{T:Kolmogorovboundmixednormal}, we define $\mathcal{C}_{bbd}$ as the set of bounded, piecewise continuously differentiable functions $g:\mathbb{R}\to\mathbb{R}$ with bounded derivative. By \eqref{E:generalsteineqsolutionbounds}, $g_h\in\mathcal{C}_{bbd}$ for each $h\in\mathcal{H}_1$. As before, we obtain:
\begin{equation} \label{E:expectedconditionalsteinidentity}
\mathbb{E}\bigl(\mathbb{V}(X|\mathscr{G}) g'(X) + \mathbb{E}(X|\mathscr{G})g(X)\bigr) = \mathbb{E}\bigl(Xg(X)\bigr) \qquad\forall g\in\mathcal{C}_{bbd}.
\end{equation}
By definition, the Wasserstein distance can be expressed as follows:
\[d_W\bigl(\mathscr{L}\bigl(\frac{X-\mu}{\sigma}\bigr),\textnormal{N$(0,1)$}\bigr) = \sup_{h\in\mathcal{H}_1}\bigl|\mathbb{E}h\bigl(\frac{X-\mu}{\sigma}\bigr) - \mathbb{E}(h(Z_{0,1}))\bigr|,\]
where, using \eqref{E:generalnormalsteineq} and \eqref{E:expectedconditionalsteinidentity},
\begin{equation} \label{E:expectedconditionalsteinrawbound}
\begin{split}
\mathbb{E}\bigl(h(\frac{X-\mu}{\sigma})\bigr) - \mathbb{E}(h(Z_{0,1})) &= \mathbb{E}\bigl(\sigma^2 g_{h}'(X) - (X-\mu) g_{h}(X)\bigr)\\
= \mathbb{E}\bigl((\sigma^2-\mathbb{V}(X|\mathscr{G})) g_{h}'(X) &+ (\mu-\mathbb{E}(X|\mathscr{G}))) g_{h}(X)\bigr)\qquad\forall h\in\mathcal{H}_1.
\end{split}
\end{equation}
If we choose $\mu =\mathbb{E}(X)$ and $\sigma^2 = \mathbb{E}\bigl(\mathbb{V}(X|\mathscr{G})\bigr)$, the second term on the right hand side of \eqref{E:expectedconditionalsteinrawbound} can be handled in the same way as in the proof of Theorem~\ref{T:Kolmogorovboundmixednormal}, yielding the bound
\[\mathbb{E}\bigl((\mu-\mathbb{E}(X|\mathscr{G})) g_h(X)\bigr) \leq 4\frac{\mathbb{V}\bigl(\mathbb{E}(X|\mathscr{G})\bigr)}{\mathbb{E}\bigl(\mathbb{V}(X|\mathscr{G})\bigr)} + 4\sqrt{\frac{2}{\pi}}\frac{\sqrt{\mathbb{V}\bigl(\mathbb{E}(X|\mathscr{G})\bigr)}\mathbb{V}\bigl(\mathbb{V}(X|\mathscr{G})\bigr)^{1/4}}{\mathbb{E}\bigl(\mathbb{V}(X|\mathscr{G})\bigr)}\quad\forall h\in\mathcal{H}_1.\]
For the first term on the right hand side of \eqref{E:expectedconditionalsteinrawbound}, letting the random variable $Y\sim\textnormal{N}(\mu,\sigma^2)$ be independent of $\mathscr{G}$, we can write:
\[\mathbb{E}\bigl((\sigma^2-\mathbb{V}(X|\mathscr{G})) g'_h(X)\bigr) = \mathbb{E}\bigl((\sigma^2-\mathbb{V}(X|\mathscr{G}))(g'_h(X) - g'_h(Y))\bigr)\]
\[= \mathbb{E}\bigl((\sigma^2-\mathbb{V}(X|\mathscr{G}))I_A\mathbb{E}(g'_h(X) - g'_h(Y)|\mathscr{G})\bigr)\]
\[+ \mathbb{E}\bigl((\sigma^2-\mathbb{V}(X|\mathscr{G}))I_{A^c}\mathbb{E}(g'_h(X) - g'_h(Y)|\mathscr{G})\bigr) \qquad\forall h\in\mathcal{H}_1,\]
where $A=\{\sigma^2 \leq\mathbb{V}(X|\mathscr{G})\}$. 
We can now use exactly the same coupling as for the second term on the right hand side of \eqref{E:expectedconditionalsteinrawbound}, together with the fact that $\lVert g''_z\rVert \leq \frac{2}{\sigma^3}$, to obtain, after some calculations:
\[\bigl|\mathbb{E}\bigl((\sigma^2-\mathbb{V}(X|\mathscr{G})) g'_h(X)\bigr)\bigr|\]
\[\leq \frac{2}{\sigma^3}\mathbb{E}\bigl(\bigl|\sigma^2-\mathbb{V}(X|\mathscr{G})\bigr|\sqrt{\frac{2}{\pi}\bigl|\sigma^2-\mathbb{V}(X|\mathscr{G})\bigr|\bigr)} + \frac{2}{\sigma^3}\mathbb{E}\bigl(\bigl|\sigma^2-\mathbb{V}(X|\mathscr{G})\bigr|\bigl|\mu-\mathbb{E}(X|\mathscr{G})\bigr|\bigr)\]
\[\leq \sqrt{\frac{2}{\pi}}\frac{2}{\sigma^3}\mathbb{E}\bigl(\bigl|\sigma^2-\mathbb{V}(X|\mathscr{G})\bigr|^{3/2}\bigr) + \frac{2}{\sigma^3}\sqrt{\mathbb{E}\bigl((\mu-\mathbb{E}(X|\mathscr{G}))^2\bigr)}\sqrt{\mathbb{E}\bigl((\sigma^2-\mathbb{V}(X|\mathscr{G}))^2\bigr)}\]
\[\leq 2\sqrt{\frac{2}{\pi}}\,\frac{\mathbb{V}\bigl(\mathbb{V}(X|\mathscr{G})\bigr)^{3/4}}{\mathbb{E}\bigl(\mathbb{V}(X|\mathscr{G})\bigr)^{3/2}} + \frac{2\sqrt{\mathbb{V}\bigl(\mathbb{E}(X|\mathscr{G})\bigr)}\sqrt{\mathbb{V}\bigl(\mathbb{V}(X|\mathscr{G})\bigr)}}{\mathbb{E}\bigl(\mathbb{V}(X|\mathscr{G})\bigr)^{3/2}}\qquad\forall h\in\mathcal{H}_1.\]
\end{proof}

\begin{rem} \label{R:Wassersteinboundtwonormals}
In the case when $\mathbb{E}(X|\mathscr{G})\equiv m$ and $\mathbb{V}(X|\mathscr{G})\equiv \tau^2$ for deterministic constants $m\in\mathbb{R}$ and $\tau>0$, we obtain from \eqref{E:expectedconditionalsteinrawbound} and \eqref{E:generalsteineqsolutionbounds},
\[d_W\bigl(\textnormal{N$(m,\tau^2)$},\textnormal{N$(\mu,\sigma^2)$}\bigr) \leq \frac{4}{\sigma^2}|\sigma^2-\tau^2| + \frac{2}{\sigma}|\mu-m|.\]
\end{rem}

Finally, we point out that it is possible to derive lower bounds for the Kolmogorov and Wasserstein distances under the same assumptions as in Theorems~\ref{T:Kolmogorovboundmixednormal} and \ref{T:Wassersteinboundmixednormal}. Using ideas introduced in \cite{BarbourHall1984} (see also Chapter 3 in \cite{BarbourHolstJanson1992}), we state and derive lower bounds in the Appendix (Theorem~\ref{T:Kolmogorovlowerboundmixednormal}; the bounds for the two distances are identical apart from a constant factor). It can be seen from Theorem~\ref{T:Kolmogorovlowerboundmixednormal} that under mild conditions on the asymptotics of the higher order moments $\mathbb{E}((\mu-\mathbb{E}(X|\mathscr{G}))^4)$ and $\mathbb{E}(|\mathbb{V}(X|\mathscr{G})-\sigma^2|(\mu-\mathbb{E}(X|\mathscr{G}))^2)$, the upper bounds in Theorems~\ref{T:Kolmogorovboundmixednormal} and \ref{T:Wassersteinboundmixednormal} leaves little room for improvement. In particular, the term $\frac{\mathbb{V}\bigl(\mathbb{E}(X|\mathscr{G})\bigr)}{\mathbb{E}\bigl(\mathbb{V}(X|\mathscr{G})\bigr)}$ cannot be replaced by another that converges faster to 0. However, the lower bound would allow for $\frac{\mathbb{V}\bigl(\mathbb{V}(X|\mathscr{G})\bigr)^{3/4}}{\mathbb{E}\bigl(\mathbb{V}(X|\mathscr{G})\bigr)^{3/2}}$ to be replaced by $\frac{\mathbb{V}\bigl(\mathbb{V}(X|\mathscr{G})\bigr)}{\mathbb{E}\bigl(\mathbb{V}(X|\mathscr{G})\bigr)}$ (times some constant) in the first term, should this turn out to be possible.

\section{The Yule-Ornstein-Uhlenbeck model}\label{sec:applicationtoyoumodel}
In order to apply the results in Section~\ref{sec:normalapproxnormalmixture} to the YOU model, we first need to condition on an appropriate
$\sigma$-algebra, and then obtain formul\ae, along with their asymptotic behaviours, for
the means and variances of the conditional means and variances. 
Since the OU process is Gaussian, conditionally on the phylogeny the values of the traits at the $n$ leaves will have an $n$-dimensional 
Gaussian distribution. Hence, the natural $\sigma$-algebra to
condition on is the $\sigma$-algebra generated by the pure birth tree. For a tree with $n$ leaves,
denote this $\sigma$-algebra by $\mathcal{Y}_{n}$. Moreover, we use the following notation: $\Gamma(\cdot)$ is the gamma function, $H_n = 1 + \frac{1}{2} + \ldots + \frac{1}{n}$, and
\[b_{n,x} = {1\over 1+x}\cdot{2\over 2+x}\cdot \ldots \cdot{n\over n+x}={\Gamma(n+1) \Gamma(x+1)\over  \Gamma(n+x+1)},\quad x>-1.\]
\begin{thm}\label{T:YOU}
Consider the YOU model with $\alpha \ge 1/2$. Let $\overline{X}_{n}$ be the average value of the traits at the $n$ leaves, let $\overline{Y}_{n} = \overline{X}_{n}\sqrt{\frac{2\alpha}{\sigma_{a}^{2}}}$, and let $\delta = X(0)\sqrt{\frac{2\alpha}{\sigma_{a}^{2}}}$. Let also $\mu_n = \mathbb{E}(\overline{Y}_{n})$ and $\sigma_n^2 = \mathbb{E}(\mathbb{V}(\overline{Y}_{n}|\mathscr{G}))$.

(i)  If $\alpha=\frac{1}{2}$, then: $d_{K}\bigl(\mathcal{L}\left(\frac{\overline{Y}_{n}-\mu_n}{\sigma_n}\right),\mathrm{N}(0,1)\bigr)= \textnormal{O}(\ln^{-1} n)$ as $n\to\infty$, 
where $\mu_n = \delta b_{n,1/2}$ and $\sigma_n^2 = \frac{1}{n} + (1-\frac{1}{n})\Bigl(\frac{2}{n-1}(H_n-1)-\frac{1}{n-1}\Bigr) - b_{n,1}$.
Moreover, $(\frac{n}{\ln n})^{1/2}\,\mu_n\to 0$ and $\frac{n}{\ln n}\,\sigma_n^2\to 2$ as $n\to\infty$, so $(\frac{n}{\ln n})^{1/2}\,\overline{Y}_{n}\ \xrightarrow{\ d\ }\ \textnormal{N}(0,2)$ as $n\to\infty$.

(ii) If $\alpha>\frac{1}{2}$, then: $d_{K}\bigl(\mathcal{L}\left(\frac{\overline{Y}_{n}-\mu_n}{\sigma_n}\right),\mathrm{N}(0,1)\bigr) = \begin{cases}
\textnormal{O}(n^{-2\alpha+1}),&\textnormal{$\frac{1}{2}<\alpha<\frac{3}{4}$;}\\
\textnormal{O}(\frac{\ln^{1/2}n}{n^{1/2}}),&\textnormal{$\alpha=\frac{3}{4}$;}\\
\textnormal{O}(n^{-1/2}),&\textnormal{$\alpha>\frac{3}{4}$,}
\end{cases}$\\
as $n\to\infty$, where $\mu_n = \delta b_{n,\alpha}$, and $\sigma_n^2 = \frac{1}{n} + (1-\frac{1}{n})\Bigl(\frac{2 - (n+1)(2\alpha + 1)b_{n,2\alpha}}{(n-1)(2\alpha-1)}\Bigr) - b_{n,2\alpha}$.
Moreover, $n^{1/2}\,\mu_n\to 0$ and $n\sigma_n^2\to\frac{2\alpha+1}{2\alpha-1}$ as $n\to\infty$, so $n^{1/2}\,\;\overline{Y}_{n}\ \xrightarrow{\ d\ }\ \textnormal{N}(0,\frac{2\alpha+1}{2\alpha-1})$ as $n\to\infty$.
\end{thm}

\begin{thm}\label{T:YOUWasserstein}
Consider the YOU model with $\alpha \ge 1/2$, with the same notation as in Theorem~\ref{T:YOU}.

(i)  If $\alpha=\frac{1}{2}$, then: $d_{W}\bigl(\mathcal{L}\left(\frac{\overline{Y}_{n}-\mu_n}{\sigma_n}\right),\mathrm{N}(0,1)\bigr)= \textnormal{O}(\ln^{-1} n)$ as $n\to\infty$.

(ii) If $\alpha>\frac{1}{2}$, then: $d_{W}\bigl(\mathcal{L}\left(\frac{\overline{Y}_{n}-\mu_n}{\sigma_n}\right),\mathrm{N}(0,1)\bigr) = \begin{cases}
\textnormal{O}(n^{-2\alpha+1}),&\textnormal{$\frac{1}{2}<\alpha<\frac{3}{4}$;}\\
\textnormal{O}(\frac{\ln^{1/4}n}{n^{1/2}}),&\textnormal{$\alpha=\frac{3}{4}$;}\\
\textnormal{O}(n^{-\min(\alpha-1/4,3/4)}),&\textnormal{$\alpha>\frac{3}{4}$,}
\end{cases}$\\
as $n\to\infty$.
\end{thm}

\begin{proof}[Proof of Theorems 3.1--2]
As explained above, the phylogeny is modelled by a pure birth tree, in which each speciation point is binary, and the edge lengths are independent exponentially distributed random variables with the same rate parameter, called the birth rate. Without loss of generality we take $1$ as the birth rate. Then, the time between the $k$th and 
$(k+1)$st speciation event, denoted $T_{k+1}$, is exponentially distributed with rate $(k+1)$, as the minimum
of $(k+1)$ independent rate 1 exponentially distributed random variables; see Fig.~\ref{F:YuleTree}.

There are two key random components to consider: the height of the tree $(U_{n})$ and the time from the present
backwards to the coalescence of a \emph{random} (out of $\binom{n}{2}$ possible) pair of tip species $(\tau^{(n)})$.
These random variables are illustrated in Fig.~\ref{F:YuleTree}, but see also Fig.~A.$8$ in \cite{Bartoszek2014} 
and Figs.~1 and~5 in \cite{Bartoszek2020}. 

\begin{figure}
\begin{center}
\includegraphics[width=0.45\textwidth]{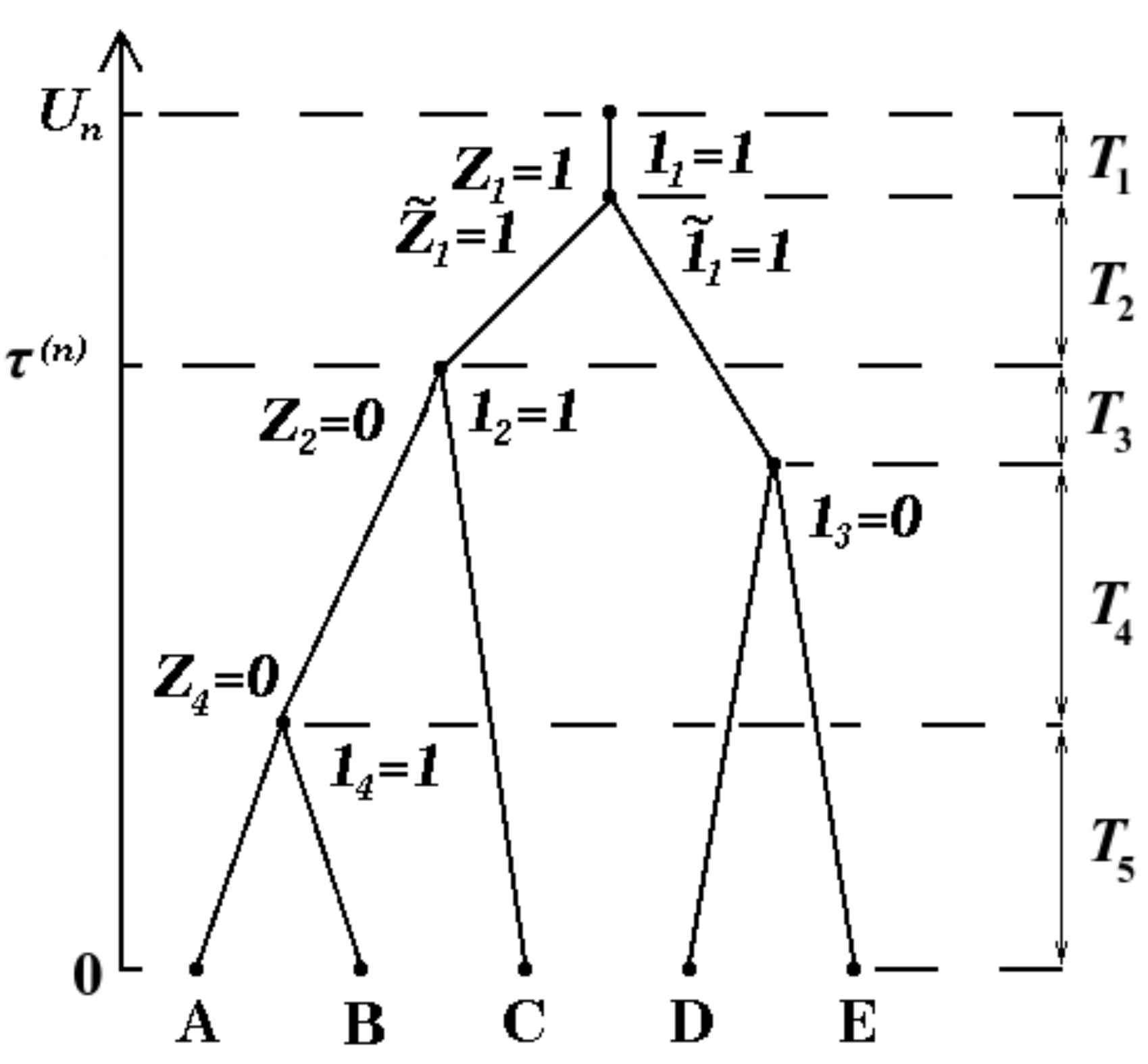}
\caption{
A pure-birth (Yule) tree with the various time components marked on it. A branching OU process, which might also have a jump just after each speciation event (=branching point), evolves on top of the tree. In this example we assume that a jump only takes place just after the first speciation event.--- The values of $\mathbf{1}_{1}$, $\mathbf{1}_{2}$, $\mathbf{1}_{3}$, 
$\mathbf{1}_{4}$, $Z_{1}$, $Z_{2}$ and $Z_{3}$ refer to the situation where node $A$ is randomly sampled.
The $\mathbf{1}_{i}$ random variables tell us if the $i$th speciation event is on the selected lineage,
while the $Z_{i}$ variables tell us if a jump took place on the lineage just after the $i$th speciation event.
As the third speciation event does not lie on the lineage to node $A$, $Z_{3}$ is undefined.
The values of $\tilde{\mathbf{1}}_{1}$, $\tilde{Z}_{1}$ and $\tau^{(n)}$ refer to the situation
where the pair of nodes $(A,C)$ was randomly sampled. As jumps take place after speciation
events the only common jump possibility for this pair is at speciation node $1$. Hence $\tilde{\mathbf{1}}_{i}$, $\tilde{Z}_{i}$ for $i>1$ are undefined.
}\label{F:YuleTree}
\end{center}
\end{figure}

In order to study the properties of the OU (and, in the next section, OU+jumps) process evolving on a tree, we need expressions for the Laplace transforms of the above random objects
that contribute to the mean and variance of the
average of the tip values, $\overline{X}_{n}$. In \cite{BartoszekSagitov2015} the following formul\ae, including the asymptotic
behaviour as $n\to\infty$, are derived (their Lemmata 3 and 4):
\be\label{E:LapUn}
\begin{array}{ll}
\Expectation{e^{-xU_n}} &= b_{n,x}\sim \Gamma(x+1)n^{-x}, \\
\V{e^{-xU_n}} &=  b_{n,2x}-b_{n,x}^2\sim (\Gamma(2x+1)-\Gamma(x+1)^{2})n^{-2x},
\end{array}
\ee
\begin{align}\label{E:Laptaun}
\Expectation{e^{-y\tau^{(n)}}}  =
\left\{
\begin{array}{lll}
\frac{2}{n-1}(H_{n}-1)-\frac{1}{n+1} &\sim\ \ 2n^{-1}\ln n,& y=1,\\
{2-(n+1)(y+1)b_{n,y}\over(n-1)(y-1)} &\sim\ \ {2\over y-1}n^{-1},& y>1,
\end{array}
\right.
\end{align}
\begin{align} \label{E:LapUnmtaun}
\Expectation{e^{-xU_n-y\tau^{(n)}}} &\sim \left\{ 
\begin{array}{ll}
2\Gamma(x+1)n^{-x-1}\ln n,& y=1,\\
{2\Gamma(x+1)\over y-1} n^{-x-1},& y>1,
\end{array}
\right.
\end{align}
as well as the variance of the conditional expectation (cf.~Lemmata $5.1$ in \cite{Bartoszek2020} and $11$ in \cite{BartoszekSagitov2015}): 
\be\label{E:VEtLapTaun}
\V{\Et{e^{-y\tau^{(n)}}}} = 
\left\{
\begin{array}{cc}
O(n^{-2y}) & 0 < y < \frac{3}{2}, \\
O(n^{-3}\ln n) & y =\frac{3}{2}, \\
O(n^{-3}) & y>\frac{3}{2}.
\end{array}
\right.
\ee
We furthermore have (Lemma $8$ in \cite{BartoszekSagitov2015}):
\be\label{E:EVYnYOU}
\begin{array}{ll}
\Et{\overline{Y}_{n}}  &=\delta e^{-\alpha U_{n}}, \\
\Vart{\overline{Y}_{n}}  &=  n^{-1}+(1-n^{-1})\Et{e^{-2\alpha \tau^{(n)}}} - e^{-2\alpha U_{n}},
\end{array}
\ee
\be\label{E:EVtYnYOUasympt}
\begin{array}{ll}
\Expectation{\Vart{\overline{Y}_{n}}} =
 n^{-1} + &(1-n^{-1})\Expectation{e^{-2\alpha \tau^{(n)}}} - \Expectation{e^{-2\alpha U_{n}}}\\
&\sim \left\{
\begin{array}{cc}
2n^{-1}\ln n, & \alpha=1/2,\\
\frac{2\alpha+1}{2\alpha-1}n^{-1},
& \alpha>1/2
\end{array}
\right.
\end{array}
\ee
and (Lemma $4$ in \cite{BartoszekSagitov2015})
\be \label{E:VEtYnYOUasympt}
\V{\Et{\overline{Y}_{n}}} = \V{\delta e^{-\alpha U_{n}}} \sim \delta^{2}(\Gamma(2\alpha+1)-\Gamma(\alpha+1)^{2})n^{-2\alpha}.
\ee
It remains to consider $\mathbb{V}(\mathbb{V}(\overline{Y}_{n} \vert \mathcal{Y}_{n}))$. Using \eqref{E:EVYnYOU}, we obtain:
\[\V{\Vart{\overline{Y}_{n}}}  =  \V{(1-n^{-1})\Et{e^{-2\alpha \tau^{(n)}}} - e^{-2\alpha U_{n}}}\]
\[= (1-n^{-1})^{2}\V{\Et{e^{-2\alpha \tau^{(n)}}}} + \V{e^{-2\alpha U_{n}}}
-2(1-n^{-1})\covx{\Et{e^{-2\alpha \tau^{(n)}}}}{e^{-2\alpha U_{n}}}\]
\[= (1-n^{-1})^{2}\V{\Et{e^{-2\alpha \tau^{(n)}}}} + \V{e^{-2\alpha U_{n}}}\]
\[- 2(1-n^{-1})\left(\Expectation{e^{-2\alpha(\tau^{(n)}+U_{n})}}-\Expectation{e^{-2\alpha \tau^{(n)}}}\Expectation{e^{-2\alpha U_{n}}}\right).\]
We consider the $\alpha \ge 1/2$ regime. 
As normality of the limiting distribution was not shown for $\alpha <1/2$ in~\cite{BartoszekSagitov2015} (and should not be expected, see Remark~\ref{R:alphalessthanahalf} below), there will be no gain from presenting long formul\ae\ for that case.
Using \eqref{E:LapUn}, \eqref{E:Laptaun} and \eqref{E:LapUnmtaun} (see also Lemmata $3$ and $4$ in \cite{BartoszekSagitov2015} 
and Lemma $5.1$ in \cite{Bartoszek2020}), and, when considering $\V{\Et{e^{-2\alpha \tau^{(n)}}}}$, using the approximation 
for large $n$
\be\label{E:sumirasympt}
\sum\limits_{i=k}^{n}i^{r} \sim
\left \{
\begin{array}{ll}
\frac{1}{r+1} (n^{r+1}-k^{r+1}), & r > -1 \\
\ln n, & r = -1 \\
\frac{1}{r+1} (k^{r+1} - n^{r+1}), & r < -1
\end{array}
\right.
\ee
due from
$$
\int\limits_{k+1}^{n-1} x^{r} \ud x \le \sum\limits_{i=k}^{n}i^{r} \le \int\limits_{k-1}^{n+1} x^{r} \ud x,
$$
we obtain the following asymptotic behaviour as $n\to\infty$:
\be\label{E:eqVVYOU}
\V{\Vart{\overline{Y}_{n}}}
\sim\left\{
\begin{array}{ll}

8\zeta_{2}n^{-2} + (\Gamma(3)-\Gamma(2))^{2}n^{-2},
& \alpha = \frac{1}{2}, \\~\\

\begin{split}
&\frac{32\alpha^{2}}{2-2\alpha}\zeta_{4-4\alpha}n^{-4\alpha}\\
&+ (\Gamma(4\alpha+1)-\Gamma(2\alpha+1))^{2}n^{-4\alpha},
\end{split} & \frac{1}{2} < \alpha < \frac{3}{4}, \\~\\
36n^{-3}\ln n,
& \alpha  = \frac{3}{4}, \\~\\

\frac{32\alpha^{2}}{(2\alpha-1)(4\alpha-3)(4\alpha-2)}n^{-3},
& \frac{3}{4} < \alpha < 1, \\~\\ 

16n^{-3},
& \alpha = 1, \\~\\

\frac{32\alpha^{2}}{(4\alpha-3)(4\alpha-2)(2\alpha-1)}n^{-3},
& 1 < \alpha,
\end{array}
\right.
\ee
where $\zeta_{r}$ is the Riemann zeta function, 
$$
\zeta_{r}=\sum\limits_{k=1}^{\infty} k^{-r}.
$$
Denote now the leading constant of $\mathbb{E}(\mathbb{V}(\overline{Y}_{n} \vert \mathcal{Y}_{n}))$
as $C^{EV}_{a,b}$, of $\mathbb{V}(\mathbb{E}(\overline{Y}_{n} \vert \mathcal{Y}_{n}))$ as $C^{VE}$,
of $\mathbb{V}(\mathbb{V}(\overline{Y}_{n} \vert \mathcal{Y}_{n}))$ as $C^{VV}_{a,b}$, 
where $a,b$ is the interval where $\alpha$ belongs to. If $a=b$, then we just write $C^{VV}_{a}$.
We drop in the notation the dependence of the constant on $\alpha$ and $X(0)$, 
treating them as implied. For $\alpha=\frac{1}{2}$, Theorem~\ref{T:Kolmogorovboundmixednormal} gives:
\be \label{E:dKYOUcrit}
\begin{array}{l}
d_{K}\bigl(\mathcal{L}\left(\frac{\overline{Y}_{n}-\mu_n}{\sigma_n}\right),\mathrm{N}(0,1)\bigr)\\
\le \frac{\sqrt{\mathbb{V}\bigl(\mathbb{V}(\overline{Y}_{n}|\mathcal{Y}_n)\bigr)}}{\mathbb{E}\bigl(\mathbb{V}(\overline{Y}_{n}|\mathcal{Y}_n)\bigr)} + \frac{\mathbb{V}\bigl(\mathbb{E}(\overline{Y}_{n}|\mathcal{Y}_n)\bigr)}{\mathbb{E}\bigl(\mathbb{V}(\overline{Y}_{n}|\mathcal{Y}_n)\bigr)} + \sqrt{\frac{2}{\pi}}\frac{\sqrt{\mathbb{V}\bigl(\mathbb{E}(\overline{Y}_{n}|\mathcal{Y}_n)\bigr)}\mathbb{V}\bigl(\mathbb{V}(\overline{Y}_{n}|\mathcal{Y}_n)\bigr)^{1/4}}{\mathbb{E}\bigl(\mathbb{V}(\overline{Y}_{n}|\mathcal{Y}_n)\bigr)}
\\
\lesssim
\frac{\sqrt{C^{VV}_{1/2}}}{C^{EV}_{1/2}}\ln^{-1}n
+
\frac{C^{VE}}{C^{EV}_{1/2}}\ln^{-1}n
+
\sqrt{\frac{2}{\pi}}\frac{\sqrt{C^{VE}}(C^{VV}_{1/2})^{1/4}}{C^{EV}_{1/2}}
\ln^{-1}n
\end{array}
\ee
where $\mu_n$ and $\sigma_n^2$, as well as their asymptotic behaviour as $n\to\infty$, 
can be obtained from \eqref{E:LapUn}, \eqref{E:Laptaun}, \eqref{E:EVYnYOU}, 
and \eqref{E:EVtYnYOUasympt}. It follows immediately from \eqref{E:Kolmogorovscaleinvariance} 
and Remark~\ref{R:Kolmogorovboundtwonormals} 
that $(\frac{n}{\ln n})^{1/2}\,\overline{Y}_{n}\ \xrightarrow{\ d\ }\ \textnormal{N}(0,2)$ as $n\to\infty$. 
Analogously, for $\alpha>\frac{1}{2}$, Theorem~\ref{T:Kolmogorovboundmixednormal} gives:

\be \label{E:dKYOUfast}
\begin{array}{l}
d_{K}\bigl(\mathcal{L}\left(\frac{\overline{Y}_{n}-\mu_n}{\sigma_n}\right),\mathrm{N}(0,1)\bigr)\\
\lesssim
\left\{
\begin{array}{ll}
\frac{\sqrt{C^{VV}_{1/2,3/4}}}{C^{EV}_{1/2,\infty}}n^{-2\alpha+1}
+ \frac{C^{VE}}{C^{EV}_{1/2,\infty}}n^{-2\alpha+1}
+ \sqrt{\frac{2}{\pi}}\frac{\sqrt{C^{VE}}(C^{VV}_{1/2,3/4})^{1/4}}{C^{EV}_{1/2,\infty}}n^{-2\alpha+1}
& \frac{1}{2} < \alpha < \frac{3}{4}, \\~\\ 
\frac{\sqrt{C^{VV}_{3/4}}}{C^{EV}_{1/2,\infty}}\frac{\ln^{1/2}n}{n^{1/2}}
+ \frac{C^{VE}}{C^{EV}_{1/2,\infty}}n^{-1/2}
+ \sqrt{\frac{2}{\pi}}\frac{\sqrt{C^{VE}}(C^{VV}_{3/4})^{1/4}}{C^{EV}_{1/2,\infty}}\frac{\ln^{1/4}n}{n^{1/2}}
& \alpha = \frac{3}{4}, \\~\\ 
\frac{\sqrt{C^{VV}_{3/4,1}}}{C^{EV}_{1/2,\infty}}n^{-1/2}
+ \frac{C^{VE}}{C^{EV}_{1/2,\infty}}n^{-2\alpha + 1}
+ \sqrt{\frac{2}{\pi}}\frac{\sqrt{C^{VE}}(C^{VV}_{3/4,1})^{1/4}}{C^{EV}_{1/2,\infty}}n^{-\alpha+1/4}
& \frac{3}{4} < \alpha < 1, \\~\\ 
\frac{\sqrt{C^{VV}_{1}}}{C^{EV}_{1/2,\infty}}n^{-1/2}
+ \frac{C^{VE}}{C^{EV}_{1/2,\infty}}n^{-1}
+ \sqrt{\frac{2}{\pi}}\frac{\sqrt{C^{VE}}(C^{VV}_{1})^{1/4}}{C^{EV}_{1/2,\infty}}n^{-3/4}
& \alpha = 1, \\~\\ 
\frac{\sqrt{C^{VV}_{1,\infty}}}{C^{EV}_{1/2,\infty}}n^{-1/2}
+ \frac{C^{VE}}{C^{EV}_{1/2,\infty}}n^{-2\alpha + 1}
+ \sqrt{\frac{2}{\pi}}\frac{\sqrt{C^{VE}}(C^{VV}_{1,\infty})^{1/4}}{C^{EV}_{1/2,\infty}}n^{-\alpha+1/4}
& 1 < \alpha.
\end{array}
\right.
\end{array}
\ee
We obtain $\mu_n$ and $\sigma_n^2$, their asymptotic behaviour as $n\to\infty$, and the fact that $n^{1/2}\,\overline{Y}_{n}\ \xrightarrow{\ d\ }\ \textnormal{N}(0,\frac{2\alpha+1}{2\alpha-1})$ as $n\to\infty$, just as in the previous case. 

For the Wasserstein distance, the first term on the right hand side of \eqref{E:dKYOUcrit} should be replaced by:
\be \label{E:dWYOUcrit}
\begin{array}{l}
\sqrt{\frac{2}{\pi}}\,\frac{\mathbb{V}\bigl(\mathbb{V}(\overline{Y}_{n}|\mathcal{Y}_n)\bigr)^{3/4}}{\mathbb{E}\bigl(\mathbb{V}(\overline{Y}_{n}|\mathcal{Y}_n)\bigr)^{3/2}} + \frac{\sqrt{\mathbb{V}\bigl(\mathbb{E}(\overline{Y}_{n}|\mathcal{Y}_n)\bigr)}\sqrt{\mathbb{V}\bigl(\mathbb{V}(\overline{Y}_{n}|\mathcal{Y}_n)\bigr)}}{\mathbb{E}\bigl(\mathbb{V}(\overline{Y}_{n}|\mathcal{Y}_n)\bigr)^{3/2}}\\
\lesssim \sqrt{\frac{2}{\pi}}\,\frac{(C^{VV}_{1/2})^{3/4}}{(C^{EV}_{1/2})^{3/2}}\ln^{-3/2}n + \frac{\sqrt{C^{VE}}\sqrt{C^{VV}_{1/2}}}{(C^{EV}_{1/2})^{3/2}}\ln^{-3/2}n.
\end{array}
\ee
and the first term on the right hand side of \eqref{E:dKYOUfast} should be replaced by:
\be
\begin{array}{l}
\sqrt{\frac{2}{\pi}}\,\frac{\mathbb{V}\bigl(\mathbb{V}(\overline{Y}_{n}|\mathcal{Y}_n)\bigr)^{3/4}}{\mathbb{E}\bigl(\mathbb{V}(\overline{Y}_{n}|\mathcal{Y}_n)\bigr)^{3/2}} + \frac{\sqrt{\mathbb{V}\bigl(\mathbb{E}(\overline{Y}_{n}|\mathcal{Y}_n)\bigr)}\sqrt{\mathbb{V}\bigl(\mathbb{V}(\overline{Y}_{n}|\mathcal{Y}_n)\bigr)}}{\mathbb{E}\bigl(\mathbb{V}(\overline{Y}_{n}|\mathcal{Y}_n)\bigr)^{3/2}}\\
\lesssim
\left\{
\begin{array}{ll}
\sqrt{\frac{2}{\pi}}\,\frac{(C^{VV}_{1/2,3/4})^{3/4}}{(C^{EV}_{1/2,\infty})^{3/2}}n^{-3\alpha + 3/2} + \frac{\sqrt{C^{VE}}\sqrt{C^{VV}_{1/2,3/4}}}{(C^{EV}_{1/2,\infty})^{3/2}}n^{-3\alpha + 3/2}
& \frac{1}{2} < \alpha < \frac{3}{4}, \\~\\ 
\sqrt{\frac{2}{\pi}}\,\frac{(C^{VV}_{3/4})^{3/4}}{(C^{EV}_{1/2,\infty})^{3/2}}\frac{\ln^{3/4}n}{n^{3/4}} + \frac{\sqrt{C^{VE}}\sqrt{C^{VV}_{3/4}}}{(C^{EV}_{1/2,\infty})^{3/4}}\frac{\ln^{1/2}n}{n^{3/4}}
& \alpha = \frac{3}{4}, \\~\\ 
\sqrt{\frac{2}{\pi}}\,\frac{(C^{VV}_{3/4,1})^{3/4}}{(C^{EV}_{1/2,\infty})^{3/2}}n^{-3/4} + \frac{\sqrt{C^{VE}}\sqrt{C^{VV}_{3/4,1}}}{(C^{EV}_{1/2,\infty})^{3/2}}n^{-\alpha}
& \frac{3}{4} < \alpha < 1, \\~\\ 
\sqrt{\frac{2}{\pi}}\,\frac{(C^{VV}_{1})^{3/4}}{(C^{EV}_{1/2,\infty})^{3/2}}n^{-3/4} + \frac{\sqrt{C^{VE}}\sqrt{C^{VV}_{1}}}{(C^{EV}_{1/2,\infty})^{3/2}}n^{-1}
& \alpha = 1, \\~\\ 
\sqrt{\frac{2}{\pi}}\,\frac{(C^{VV}_{1,\infty})^{3/4}}{(C^{EV}_{1/2,\infty})^{3/2}}n^{-3/4} + \frac{\sqrt{C^{VE}}\sqrt{C^{VV}_{1,\infty}}}{(C^{EV}_{1/2,\infty})^{3/2}}n^{-\alpha}
& 1 < \alpha,
\end{array}
\right.
\end{array}
\ee
\end{proof}

We illustrate the bounds from \eqref{E:dKYOUcrit} and \eqref{E:dKYOUfast} 
and for the YOUj model in Fig.~\ref{F:YOUYOUjbound}.

\begin{rem} \label{R:alphalessthanahalf}
The theorems presented in this section do not give information about the case $\alpha < 1/2$. However, one can strongly suspect that the limit will not be normal in this case. By considering higher moments of the limiting distribution, it was shown in Remark $3.14$ in \cite{AdamczakMilos2015} that when stopping the YOU model at a fixed time (the number of tips being random) for $\alpha < 1/2$, the limit is not normal. Unfortunately, when stopping just before the $n$th speciation event, the approach in \cite{BartoszekSagitov2015} does not allow for easy derivation of the higher moments, in order to reach the 
same conclusion as in \cite{AdamczakMilos2015}.
\end{rem}

\begin{figure}[!ht]
\begin{center}
\includegraphics[width=0.45\textwidth]{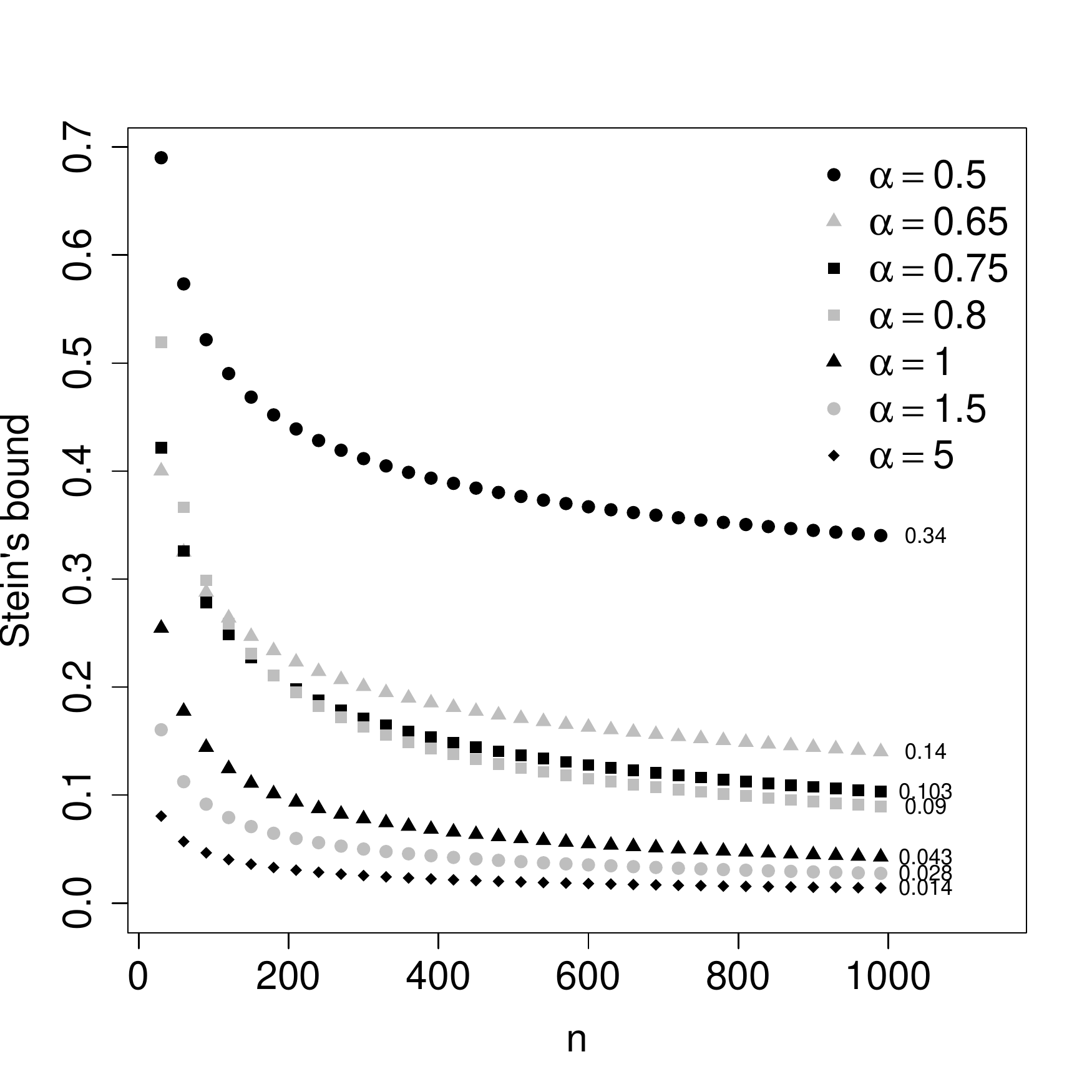}
\includegraphics[width=0.45\textwidth]{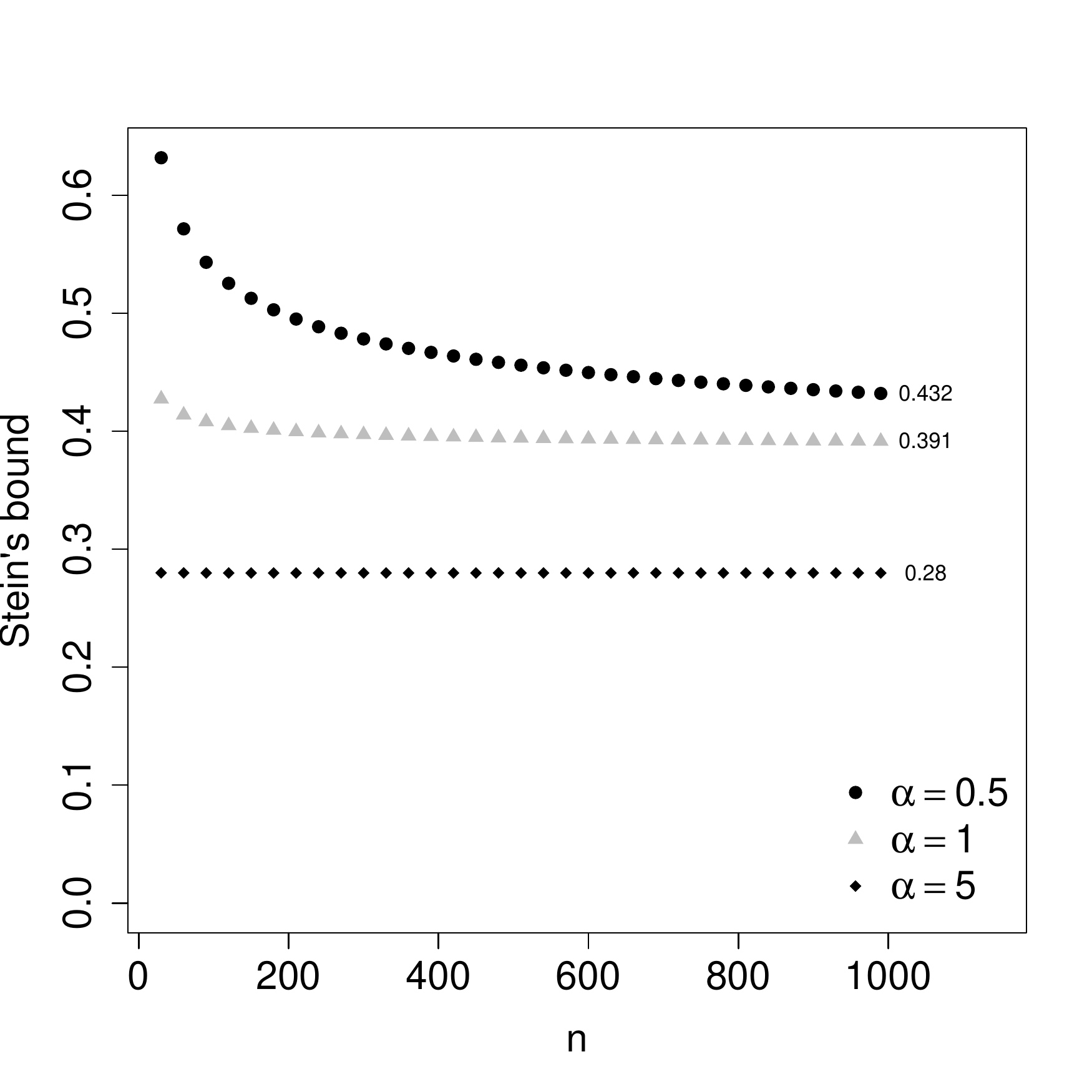}
\caption{
Left: Illustration of the bounds from \eqref{E:dKYOUcrit} and \eqref{E:dKYOUfast}. For the graph we chose $\sigma_{a}^{2}=1$ and 
$X(0)=(2\alpha)^{-1/2}$.
Right: Illustration of the bounds on the Kolmogorov distance 
for the YOUj model. For the graphs we chose $X(0)=(2\alpha)^{-1/2}$, 
$p=1/2$ and $\sigma_{a}^{2}=\sigma_{c}^{2}=1$. For $\alpha=1/2$ we use the bound of \eqref{E:dKYOUjconstaeq05}, while for $\alpha>1/2$ 
we are in the non-convergent regime, so the bounds come from
explicitly calculating the asymptotic constant in \eqref{E:dKYOUjcrit}.
}\label{F:YOUYOUjbound}
\end{center}
\end{figure}

\section{The Yule-Ornstein-Uhlenbeck model with jumps}\label{sec:applicationtoyoujumpsmodel}
The new feature of the YOUj model, as compared to the YOU model, is that a normally distributed jump with mean $0$ 
may or may not take place in the trait value immediately after a speciation event. The jumps occur independently of 
one another and of the OU process, but the probability of a jump, and the variance of the jump, may depend on the number of the speciation event: with speciation event number $i=1,\ldots,n$, we associate a jump probability $p_{i}$ and jump variance $\sigma_{c,i}^{2}$. If the jump probabilities and variances are constant, we write: $(p_i,\sigma_{c,i}^{2})\equiv (p,\sigma_{c}^{2})$.

The key problem is that one needs to keep careful track of the jumps
that take place at speciation events and how the ``mean-reversion'' of the OU process
part causes their effect to be smoothed out along a lineage. We keep the notation defined in Section~\ref{sec:applicationtoyoumodel}, except that we now denote by $\mathcal{Y}_{n}$ the $\sigma$-algebra
that contains information on the whole Yule tree and the jumps' locations, i.e. after which speciation events
did a jump take place. We now introduce the concept of convergence with density $1$.
\begin{defn}
A subset $E \subset \mathbb{N}$ of positive integers is said to have density $0$ (see e.g.~Petersen \cite{KPet1983}) if 
$$
\lim\limits_{n \to \infty}\frac{1}{n}\sum\limits_{k=0}^{n-1}I_{E}(k) =0,
$$
where $I_{E}(\cdot)$ is the indicator function of the set $E$.
\end{defn}

\begin{defn}
A sequence $a_{n}$ converges to $0$ with density $1$ if there exists a subset $E\subset \mathbb{N}$ of density $0$ such that 
$$
\lim\limits_{n \to \infty,n \notin E}a_{n} =0.
$$
\end{defn}

\begin{thm}\label{T:YOUj}
Consider the YOUj model with $\alpha \ge 1/2$. Let $\overline{X}_{n}$ be the average value of the traits at the $n$ leaves, let $\overline{Y}_{n} = \overline{X}_{n}\sqrt{\frac{2\alpha}{\sigma_{a}^{2}}}$, and let $\delta = X(0)\sqrt{\frac{2\alpha}{\sigma_{a}^{2}}}$. Let also $\mu_n = \mathbb{E}(\overline{Y}_{n})$ and $\sigma_n^2=\Expectation{\Vart{\overline{Y}_{n}}}$.

(i)  If $\alpha=1/2$, and $(p_i,\sigma_{c,i}^{2})\equiv (p,\sigma_{c}^{2})$, then: $d_{K}\bigl(\mathcal{L}\left(\frac{\overline{Y}_{n}-\mu_n}{\sigma_n}\right),\mathrm{N}(0,1)\bigr) = \textnormal{O}(\ln^{-\frac{1}{2}} n)$ as $n\to\infty$, where $\mu_n = \delta b_{n,\alpha}$. Moreover, $(\frac{n}{\ln n})^{1/2}\,\mu_n\to 0$ and $\frac{n}{\ln n}\sigma_n^2\to 2+\frac{4p}{\sigma_a^{2}}\sigma_{c}^{2}$ as $n\to\infty$, so $(\frac{n}{\ln n})^{1/2}\,\overline{Y}_{n}\ \xrightarrow{\ d\ }\ \textnormal{N}(0,2+\frac{4p}{\sigma_a^{2}}\sigma_{c}^{2})$ as $n\to\infty$.

(ii) If $\alpha>1/2$, and $(p_i,\sigma_{c,i}^{2})\equiv (1,\sigma_{c}^{2})$, then the asymptotics as $n\to\infty$ for $d_{K}\bigl(\mathcal{L}\left(\frac{\overline{Y}_{n}-\mu_n}{\sigma_n}\right),\mathrm{N}(0,1)\bigr)$ is the same as in Theorem~\ref{T:YOU} ($ii$), and $\mu_n = \delta b_{n,\alpha}$. Moreover, $n^{1/2}\,\mu_n\to 0$ and $n\sigma_n^2\to\frac{2\alpha+1}{2\alpha-1}(1 + \frac{2p}{\sigma_a^{2}}\sigma_{c}^{2})$ as $n\to\infty$, so $n^{1/2}\,\overline{Y}_{n}\ \xrightarrow{\ d\ }\ \textnormal{N}(0,\frac{2\alpha+1}{2\alpha-1}(1 + \frac{2p}{\sigma_a^{2}}\sigma_{c}^{2}))$ as $n\to\infty$.

(iii) If $\alpha>1/2$, and the sequence $p_{n}\sigma_{c,n}^{4}$ 
is bounded and converges to $0$ with density $1$, then: $d_{K}\bigl(\mathcal{L}\left(\frac{\overline{Y}_{n}-\mu_n}{\sigma_n}\right),\mathrm{N}(0,1)\bigr) \to 0$ as $n\to\infty$.
\end{thm}

\begin{thm}\label{T:YOUjWasserstein}
Consider the YOUj model with $\alpha \ge 1/2$, with the same notation as in Theorem~\ref{T:YOUj}.

(i)  If $\alpha=1/2$, and $(p_i,\sigma_{c,i}^{2})\equiv (p,\sigma_{c}^{2})$, then: $d_{K}\bigl(\mathcal{L}\left(\frac{\overline{Y}_{n}-\mu_n}{\sigma_n}\right),\mathrm{N}(0,1)\bigr) = \textnormal{O}(\ln^{-3/4} n)$ as $n\to\infty$.

(ii) If $\alpha>1/2$, and $(p_i,\sigma_{c,i}^{2})\equiv (1,\sigma_{c}^{2})$, then the asymptotics as $n\to\infty$ for $d_{K}\bigl(\mathcal{L}\left(\frac{\overline{Y}_{n}-\mu_n}{\sigma_n}\right),\mathrm{N}(0,1)\bigr)$ is the same as in Theorem~\ref{T:YOUWasserstein} ($ii$). 

(iii) If $\alpha>1/2$, and the sequence $p_{n}\sigma_{c,n}^{4}$ 
is bounded and converges to $0$ with density $1$, then: $d_{W}\bigl(\mathcal{L}\left(\frac{\overline{Y}_{n}-\mu_n}{\sigma_n}\right),\mathrm{N}(0,1)\bigr) \to 0$ as $n\to\infty$.
\end{thm}

\begin{proof}[Proof of Theorems 4.1--2]
In addition to the random quantities defined in Section~\ref{sec:applicationtoyoumodel}, we have to consider
two more random components of the tree,
the speciation events on a \emph{random} (out of $n$ possible) lineage, and
the speciation events common (i.e. on the path from the origin
of the tree to the most recent common ancestor) 
for a \emph{random} (out of $\binom{n}{2}$ possible) pair of tip species. We define $\mathbf{1}_{i}$ as a binary random variable indicating that the tree's $i$th speciation event
is present on our randomly chosen lineage, 
$\tilde{\mathbf{1}}_{i}$ as a binary random variable indicating that the tree's $i$th speciation event
is present the path from the root to the most recent common ancestor of our randomly sampled pair of tips, 
$Z_{i}$ as a binary random variable indicating that a jump took place just after the
tree's $i$th speciation event on our randomly chosen lineage
and 
$\tilde{Z}_{i}$ as a binary random variable indicating that a jump took place just after the
tree's $i$th speciation event on the path from the root to the most recent common ancestor of our randomly sampled
pair of tips.
For illustration of these random variables see Fig.~\ref{F:YuleTree}.

Furthermore, we define the two following sequences of random variables:
\[\phi^{\ast}_{i}:=Z_{i}e^{-2\alpha(T_{n}+\ldots+T_{i+1})}\Et{\mathbf{1}_{i}}; \quad\phi_{i}:=\tilde{Z}_{i}\tilde{\mathbf{1}}_{i}e^{-2\alpha(T_{n}+\ldots+T_{i+1})} \qquad\forall i=1,\ldots,n-1.\]
We can recognize that $\phi^{\ast}_{i}$ and $\phi_{i}$ 
capture how the effect of each (potential) jump will be modified before the end 
of the randomly selected lineage is reached. The first one quantifies the effects that jumps will have on a randomly selected tip species, while the second quantifies the effects that jumps have on the covariance
between a random pair of tip species. 
Intuitively speaking, a random event at distance (in time) $t$ away from the point of interest, is under the OU process
discounted by a factor of $e^{-\alpha t}$, implying that the contribution
of its variance will be discounted by $e^{-2\alpha t}$. 

Recall that with each speciation event, $i=1,\ldots,n$, we associate the jump probability $p_{i}$ and jump variance $\sigma_{c,i}^{2}$, and that the jumps are normally distributed with mean 0. In the case when $(p_i,\sigma_{c,i}^{2})\equiv (p,\sigma_{c}^{2})$, we have (the $\alpha \ge 1/2$ regime in the proof of Theorem $3.2$ in \cite{Bartoszek2014}):

\be\label{E:Ephiastsum}
\Expectation{\sum\limits_{i=1}^{n-1}\phi^{\ast}_{i}} = 
\frac{2p}{2\alpha}(1-(1+2\alpha)b_{n,2\alpha})\sim \frac{2p}{2\alpha}(1-\Gamma(2+2\alpha)n^{-2\alpha}),
\ee

\be\label{E:Ephisum}
\Expectation{\sum\limits_{i=1}^{n-1}\phi_{i}} = 
\left\{
\begin{array}{cll}
\frac{4p}{n-1}(H_{n}-\frac{5n-1}{2(n+1)}) &\sim\ \ 4pn^{-1}\ln n, & \alpha =1/2\\
\frac{2p}{2\alpha}\frac{(2-(2\alpha+1)(2\alpha n -2\alpha+2)b_{n,2\alpha})}{(n-1)(2\alpha-1)} &\sim\ \ \frac{4p}{2\alpha(2\alpha-1)}n^{-1}, & \alpha >1/2.
\end{array}
\right.
\ee
In the case when $p_{n}\sigma_{c,n}^{4}\to 0$ 
with density $1$ as $n\to\infty$, then, by Corollaries 
$5.4$ and $5.7$ in \cite{Bartoszek2020}, as $n\to\infty$, 
for $\alpha=1/2$,
\be\label{E:phiphiastasymptdens1eq05}
\begin{array}{l}
(n\ln^{-1} n) \V{\sum\limits_{i=1}^{n-1}\sigma_{c,i}^{2}\phi^{\ast}_{i}} \to 0,~~~~ 
(n^{2}\ln^{-1} n) \V{\sum\limits_{i=1}^{n-1}\sigma_{c,i}^{2}\phi_{i}} \to 0, 
\end{array}
\ee
and for $\alpha>1/2$
\be\label{E:phiphiastasymptdens1gt05}
\begin{array}{l}
n \V{\sum\limits_{i=1}^{n-1}\sigma_{c,i}^{2}\phi^{\ast}_{i}} \to 0,~~~~ 
n^{2} \V{\sum\limits_{i=1}^{n-1}\sigma_{c,i}^{2}\phi_{i}} \to 0. 
\end{array}
\ee
For the conditional mean and variance of $\overline{Y}_{n}$, the following 
formul\ae\ are provided in \cite{Bartoszek2020}, Lemma $6.1$: 
\be\label{E:EVYnYOUj}
\begin{array}{ll}
\Et{\overline{Y}_{n}}  &=\delta e^{-\alpha U_{n}}, \\
\Vart{\overline{Y}_{n}}  &=  n^{-1}+(1-n^{-1})\Et{e^{-2\alpha \tau^{(n)}}} - e^{-2\alpha U_{n}}
\\ &+n^{-1}\frac{2\alpha}{\sigma_a^{2}}\sum\limits_{i=1}^{n-1}\sigma_{c,i}^{2}\phi^{\ast}_{i} 
+(1-n^{-1})\frac{2\alpha}{\sigma_a^{2}}\sum\limits_{i=1}^{n-1}\sigma_{c,i}^{2}\phi_{i}. 
\end{array}
\ee
In the case when $(p_i,\sigma_{c,i}^{2})\equiv (p,\sigma_{c}^{2})$,
using \eqref{E:EVtYnYOUasympt}, \eqref{E:Ephiastsum} and \eqref{E:Ephisum}, we obtain

\be\label{E:EVtYnYOUjasympt}
\Expectation{\Vart{\overline{Y}_{n}}} \sim
\left\{
\begin{array}{cc}
2\left(1+\frac{2p}{\sigma_a^{2}}\sigma_{c}^{2}\right)n^{-1}\ln n, & \alpha=1/2, \\
\frac{2\alpha+1}{2\alpha-1}\left(1 + \frac{2p}{\sigma_a^{2}}\sigma_{c}^{2}\right)n^{-1}, & \alpha >1/2,
\end{array}
\right.
\ee
and as in \eqref{E:VEtYnYOUasympt}, we get:
$$
\V{\Et{\overline{Y}_{n}}} = \V{\delta e^{-\alpha U_{n}}} \sim \delta^{2}(\Gamma(2\alpha+1)-\Gamma(\alpha+1)^{2})n^{-2\alpha}.
$$
It remains to consider $\mathbb{V}(\mathbb{V}(\overline{Y}_{n} \vert \mathcal{Y}_{n}))$.
We will use Cauchy-Schwarz to obtain an upper bound

\be\label{E:VVtYnYOUjbound}
\begin{array}{ll}
\V{\Vart{\overline{Y}_{n}}} & \le 4\left(\V{\Et{e^{-2\alpha \tau^{(n)}}}} + 
\V{e^{-2\alpha U_{n}}} 
\right. \\& \left.
+n^{-2}(\frac{2\alpha}{\sigma_a^{2}})^{2}\V{\sum\limits_{i=1}^{n-1}\sigma_{c,i}^2\phi^{\ast}_{i}}
+(\frac{2\alpha}{\sigma_a^{2}})^{2}\V{\sum\limits_{i=1}^{n-1}\sigma_{c,i}^2\phi_{i}}\right).
\end{array}
\ee
As before, we first consider the case when $(p_i,\sigma_{c,i}^{2})\equiv (p,\sigma_{c}^{2})$. We look at $\V{\sum\limits_{i=1}^{n-1}\phi^{\ast}_{i}}$ by considering in more detail the elements I, II and III inside the proof of Lemma $5.3$ in \cite{Bartoszek2020}, 
to obtain:

\be\label{E:Vphiastsumasmypt}
\V{\sum\limits_{i=1}^{n-1}\phi^{\ast}_{i}} \sim 
\left\{
\begin{array}{cc}
n^{-1}\ln n, & \alpha = 1/4,\\
\frac{4p^{2}}{4\alpha-1}n^{-1}, & \alpha> 1/4.
\end{array}
\right.
\ee
In the same fashion, we look at $\V{\sum\limits_{i=1}^{n-1}\phi_{i}}$ by 
considering in more detail element III inside the proof of Lemma $5.5$ in \cite{Bartoszek2020} and using \eqref{E:sumirasympt}

\be\label{E:Vphisumasmypt}
\V{\sum\limits_{i=1}^{n-1}\phi_{i}} \sim 
\left\{
\begin{array}{cc}
16p(1-p)n^{-2}\ln n, & \alpha =1/2,\\
\frac{32p(1-p)}{(4\alpha)(4\alpha-1)(4\alpha-2)}n^{-2}, & \alpha>1/2.
\end{array}
\right.
\ee
The other elements I, II, IV and V for $\alpha \ge 1/2$ converge faster to $0$,
hence they do not contribute to the leading asymptotic behaviour. Using \eqref{E:LapUn}, \eqref{E:VEtLapTaun}, \eqref{E:Vphiastsumasmypt} and \eqref{E:Vphisumasmypt}, we obtain the bound:

\be\label{E:VVtYnYOUjasympt}
\V{\Vart{\overline{Y}_{n}}} \lesssim
\left\{
\begin{array}{cc}
4(\frac{2\alpha}{\sigma_a^{2}})^{2}\sigma_{c}^{4}16p(1-p)n^{-2}\ln n, & \alpha=1/2 \\
4(\frac{2\alpha}{\sigma_a^{2}})^{2}\sigma_{c}^{4}\frac{32p(1-p)}{(4\alpha)(4\alpha-1)(4\alpha-2)}n^{-2}, & \alpha>1/2 .
\end{array}
\right.
\ee

We denote, just as in Section~\ref{sec:applicationtoyoumodel}, the leading constant of $\mathbb{E}(\mathbb{V}(\overline{Y}_{n} \vert \mathcal{Y}_{n}))$
as $C^{EV}_{a,b}$, of $\mathbb{V}(\mathbb{E}(\overline{Y}_{n} \vert \mathcal{Y}_{n}))$ as $C^{VE}$,
of $\mathbb{V}(\mathbb{V}(\overline{Y}_{n} \vert \mathcal{Y}_{n}))$ as $C^{VV}_{a,b}$, 
where $a,b$ is the interval where $\alpha$ belongs to. If $a=b$, then we just write $C^{VV}_{a}$. If $\alpha=1/2$ and $(p_i,\sigma_{c,i}^{2})\equiv (p,\sigma_{c}^{2})$, Theorem~\ref{T:Kolmogorovboundmixednormal} gives:
\be \label{E:dKYOUjconstaeq05}
\begin{array}{l}
d_{K}\bigl(\mathcal{L}\left(\frac{\overline{Y}_{n}-\mu_n}{\sigma_n}\right),\mathrm{N}(0,1)\bigr)\\
\le \frac{\sqrt{\mathbb{V}\bigl(\mathbb{V}(\overline{Y}_{n}|\mathcal{Y}_n)\bigr)}}{\mathbb{E}\bigl(\mathbb{V}(\overline{Y}_{n}|\mathcal{Y}_n)\bigr)} + \frac{\mathbb{V}\bigl(\mathbb{E}(\overline{Y}_{n}|\mathcal{Y}_n)\bigr)}{\mathbb{E}\bigl(\mathbb{V}(\overline{Y}_{n}|\mathcal{Y}_n)\bigr)} + \sqrt{\frac{2}{\pi}}\frac{\sqrt{\mathbb{V}\bigl(\mathbb{E}(\overline{Y}_{n}|\mathcal{Y}_n)\bigr)}\mathbb{V}\bigl(\mathbb{V}(\overline{Y}_{n}|\mathcal{Y}_n)\bigr)^{1/4}}{\mathbb{E}\bigl(\mathbb{V}(\overline{Y}_{n}|\mathcal{Y}_n)\bigr)}
\\
\lesssim
\frac{\sqrt{C^{VV}_{1/2}}}{C^{EV}_{1/2}}\ln^{-1/2}n
+
\frac{C^{VE}}{C^{EV}_{1/2}}\ln^{-1}n
+
\sqrt{\frac{2}{\pi}}\frac{\sqrt{C^{VE}}(C^{VV}_{1/2})^{1/4}}{C^{EV}_{1/2}}
\ln^{-3/4}n
\end{array}
\ee
where $\mu_n$ and $\sigma_n^2$, as well as their asymptotic behaviour as $n\to\infty$, can be obtained from \eqref{E:LapUn}, \eqref{E:EVYnYOUj}, and \eqref{E:EVtYnYOUjasympt}. Just as in Section~\ref{sec:applicationtoyoumodel}, it follows that $(\frac{n}{\ln n})^{1/2}\,\overline{Y}_{n}\ \xrightarrow{\ d\ }\ \textnormal{N}(0,2+\frac{4p}{\sigma_a^{2}}\sigma_{c}^{2})$ as $n\to\infty$. For the Wasserstein distance, the first term on the right hand side of \eqref{E:dKYOUjconstaeq05} should be replaced by:
\be \label{E:dWYOUjconstaeq05}
\begin{array}{l}
\sqrt{\frac{2}{\pi}}\,\frac{\mathbb{V}\bigl(\mathbb{V}(\overline{Y}_{n}|\mathcal{Y}_n)\bigr)^{3/4}}{\mathbb{E}\bigl(\mathbb{V}(\overline{Y}_{n}|\mathcal{Y}_n)\bigr)^{3/2}} + \frac{\sqrt{\mathbb{V}\bigl(\mathbb{E}(\overline{Y}_{n}|\mathcal{Y}_n)\bigr)}\sqrt{\mathbb{V}\bigl(\mathbb{V}(\overline{Y}_{n}|\mathcal{Y}_n)\bigr)}}{\mathbb{E}\bigl(\mathbb{V}(\overline{Y}_{n}|\mathcal{Y}_n)\bigr)^{3/2}}\\
\lesssim \sqrt{\frac{2}{\pi}}\,\frac{(C^{VV}_{1/2})^{3/4}}{(C^{EV}_{1/2})^{3/2}}\ln^{-3/4}n + \frac{\sqrt{C^{VE}}\sqrt{C^{VV}_{1/2}}}{(C^{EV}_{1/2})^{3/2}}\ln^{-1}n.
\end{array}
\ee
If $\alpha>1/2$ and $(p_i,\sigma_{c,i}^{2})\equiv (p,\sigma_{c}^{2})$, where $p<1$, Theorem~\ref{T:Kolmogorovboundmixednormal} gives:
\be \label{E:dKYOUjcrit}
\begin{array}{l}
d_{K}\bigl(\mathcal{L}\left(\frac{\overline{Y}_{n}-\mu_n}{\sigma_n}\right),\mathrm{N}(0,1)\bigr)\\
\le \frac{\sqrt{\mathbb{V}\bigl(\mathbb{V}(\overline{Y}_{n}|\mathcal{Y}_n)\bigr)}}{\mathbb{E}\bigl(\mathbb{V}(\overline{Y}_{n}|\mathcal{Y}_n)\bigr)} + \frac{\mathbb{V}\bigl(\mathbb{E}(\overline{Y}_{n}|\mathcal{Y}_n)\bigr)}{\mathbb{E}\bigl(\mathbb{V}(\overline{Y}_{n}|\mathcal{Y}_n)\bigr)} + \sqrt{\frac{2}{\pi}}\frac{\sqrt{\mathbb{V}\bigl(\mathbb{E}(\overline{Y}_{n}|\mathcal{Y}_n)\bigr)}\mathbb{V}\bigl(\mathbb{V}(\overline{Y}_{n}|\mathcal{Y}_n)\bigr)^{1/4}}{\mathbb{E}\bigl(\mathbb{V}(\overline{Y}_{n}|\mathcal{Y}_n)\bigr)}
\\
\lesssim
\frac{\sqrt{C^{VV}_{1/2,\infty}}}{C^{EV}_{1/2,\infty}} + \frac{C^{VE}}{C^{EV}_{1/2,\infty}}n^{-2\alpha + 1} + \sqrt{\frac{2}{\pi}}\frac{\sqrt{C^{VE}}(C^{VV}_{1/2,\infty})^{1/4}}{C^{EV}_{1/2,\infty}}n^{-\alpha + 1/2}.
\end{array}
\ee
The bound does not converge to 0 as $n\to\infty$. The same is true for the Wasserstein distance, where the first term on the right hand side of \eqref{E:dKYOUjcrit} should be replaced by:
\be \label{E:dWYOUjcrit}
\begin{array}{l}
\sqrt{\frac{2}{\pi}}\,\frac{\mathbb{V}\bigl(\mathbb{V}(\overline{Y}_{n}|\mathcal{Y}_n)\bigr)^{3/4}}{\mathbb{E}\bigl(\mathbb{V}(\overline{Y}_{n}|\mathcal{Y}_n)\bigr)^{3/2}} + \frac{\sqrt{\mathbb{V}\bigl(\mathbb{E}(\overline{Y}_{n}|\mathcal{Y}_n)\bigr)}\sqrt{\mathbb{V}\bigl(\mathbb{V}(\overline{Y}_{n}|\mathcal{Y}_n)\bigr)}}{\mathbb{E}\bigl(\mathbb{V}(\overline{Y}_{n}|\mathcal{Y}_n)\bigr)^{3/2}}\\
\lesssim \sqrt{\frac{2}{\pi}}\,\frac{(C^{VV}_{1/2,\infty})^{3/4}}{(C^{EV}_{1/2,\infty})^{3/2}} + \frac{\sqrt{C^{VE}}\sqrt{C^{VV}_{1/2,\infty}}}{(C^{EV}_{1/2,\infty})^{3/2}}n^{-\alpha + 1/2}.
\end{array}
\ee
However, if $p=1$, the leading term in \eqref{E:VVtYnYOUjbound} vanishes, which implies the convergence to 0 in part $(ii)$ of Theorem~\ref{T:YOUj}. In order to obtain the rate of convergence, we need to look at lower order terms. They turn out to be the same as for $\V{\Et{e^{-2\alpha \tau^{(n)}}}}$, since in the $\alpha\ge 1/2$ regime all the other terms converge to 0 just as fast 
(parts I, IV, V of Lemma $5.5$ in \cite{Bartoszek2020}) 
or faster (cf.~Lemmata $5.3$, $5.5$ in \cite{Bartoszek2020}). 
Using the convergence rates presented in \eqref{E:VEtLapTaun}, \eqref{E:VEtYnYOUasympt}, and \eqref{E:EVtYnYOUjasympt}, we obtain part $(ii)$ of Theorems~\ref{T:YOUj} and \ref{T:YOUjWasserstein}.

Finally, if $p_{n}\sigma_{c,n}^{4}$ are bounded and converge to $0$ with density $1$, then by \eqref{E:phiphiastasymptdens1gt05} we obtain
$$
\begin{array}{l}
n^2 \V{\sum\limits_{i=1}^{n-1}\sigma_{c,i}^{2}\phi_{i}} \to 0\quad\textnormal{as $n\to\infty$,}
\end{array}
$$
which implies that $n^{2}\V{\Vart{\overline{Y}_{n}}}\to 0$ as $n\to\infty$, by \eqref{E:VVtYnYOUjbound}. This in turn entails convergence of both distances to 0 as $n\to\infty$, but without any information on the rate. This proves part $(iii)$ of Theorems~\ref{T:YOUj} and \ref{T:YOUjWasserstein}.
\end{proof}

\begin{rem}
In the original arXiv preprint (ArXiv accession 1602.05189) for \cite{Bartoszek2020}, it was stated that convergence to normality
in the $\alpha \ge 1/2$ regime will only take place if $\sigma_{c,n}^{4}p_{n} \to 0$ with density $1$
and is bounded. However, in \eqref{E:dKYOUjconstaeq05} above we can see that in the critical
case, $\alpha=1/2$, convergence to normality will hold even if $(p_i,\sigma_{c,i}^{2})\equiv (p,\sigma_{c}^{2})$. 
\end{rem}

\begin{rem}
The condition $p_{n}\sigma_{c,n}^{4} \to 0$ with density $1$ in Theorem~\ref{T:YOUj} can be slightly relaxed. 
Essentially the same results (with possibly different bounds) will hold if $(1-p_{n})p_{n}\sigma_{c,n}^{4} \to 0$ 
with density $1$ with additional assumptions on the jump effects on a randomly chosen lineage and for a random pair of 
sampled lineages (see Theorem~$4.6$ in \cite{Bartoszek2020}). However, introducing this here would require a significant 
amount of additional heavy notation, for no gain in the actual application of Stein's method to the YOUj model.
\end{rem}

\section{Appendix}
\begin{thm}\label{T:Kolmogorovlowerboundmixednormal}
Let $X$ be a real valued random variable such that $\mathbb{E}(X^2)<\infty$, and let $\mathscr{G}$ be a $\sigma$-algebra such that the regular conditional distribution of $X$ given $\mathscr{G}$ is normal. Define $\mu=\mathbb{E}(X)$, $\sigma^2 = \mathbb{E}(\mathbb{V}(X|\mathscr{G}))$, and
\[\kappa(x) = (\sigma^2-x)\Bigl(\Bigl(\frac{\sigma^2}{\sigma^2 + x}\Bigr)^{3/2} - \frac{1}{2^{3/2}}\Bigr) \qquad\forall x\geq 0.\]
If the asymptotic behaviour of $X$ is such that $\sigma^{-2}\mathbb{E}\bigl((\mu-\mathbb{E}(X|\mathscr{G}))^4\bigr)$ and $\sigma^{-2}\mathbb{E}\bigl((\sigma^2-\mathbb{V}(X|\mathscr{G}))_+(\mu-\mathbb{E}(X|\mathscr{G}))^2\bigr)$ converge to 0 faster than $\mathbb{V}(\mathbb{E}(X|\mathscr{G}))\Bigl[ = \mathbb{E}\bigl((\mu-\mathbb{E}(X|\mathscr{G}))^2\bigr)\Bigr]$, and $\sigma^{-2}\mathbb{E}\bigl(|\sigma^2-\mathbb{V}(X|\mathscr{G})|(\mu-\mathbb{E}(X|\mathscr{G}))^2\bigr)$ converges to 0 faster than $\mathbb{E}(\kappa(\mathbb{V}(X|\mathscr{G})))$, then,
\[d\Bigl(\mathscr{L}\bigl(\frac{X-\mu}{\sigma}\bigr),\textnormal{N}(0,1))\Bigr) \geq \frac{\bigl||T_1(X)| - |T_2(X)|\bigr|}{C\sigma^2},\]
where: $(i)$ either $d=d_K$ and $C = \int_{-\infty}^\infty|2x^3-5x|e^{-x^2/2}dx$, or $d=d_W$ and $C = \max_{x\in\mathbb{R}}|2x^3-5x|e^{-x^2/2}$; $(ii)$ $|T_1(X)| \asymp \mathbb{V}(\mathbb{E}(X|\mathscr{G}))$ and $|T_2(X)| \sim \mathbb{E}\bigl(\kappa(\mathbb{V}(X|\mathscr{G}))\bigr)$. Moreover, $\mathbb{E}\bigl(\kappa(\mathbb{V}(X|\mathscr{G}))\bigr) \leq \frac{27}{8}\sigma^{-2}\mathbb{E}((\sigma^2-\mathbb{V}(X|\mathscr{G}))^2)$.
\end{thm}

\begin{proof} Inspired by the approach of Sections 3.2--3 in \cite{BarbourHolstJanson1992}, we define the function $g:\mathbb{R}\to\mathbb{R}$ as follows:
\[g(y) = (y-\mu)\exp\bigl(-\frac{(y-\mu)^2}{2\sigma^2}\bigr) \qquad\forall y\in\mathbb{R}.\]
It is easily seen that $g$ is bounded and has a bounded and continuous derivative. Define $h:\mathbb{R}\to\mathbb{R}$ by: $h(x) = \sigma^2g'(\sigma x + \mu) - \sigma xg(\sigma x + \mu)$ for each $x\in\mathbb{R}$. This gives:
\begin{equation} \label{E:lowerboundequation}
\mathbb{E}\bigl(\sigma^2g'(X) - (X-\mu)g(X)\bigr) = \mathbb{E}(h(\frac{X-\mu}{\sigma})) - \mathbb{E}(h(\frac{Z-\mu}{\sigma})),
\end{equation}
where $Z\sim\textnormal{N}(\mu,\sigma^2)$. By the Stein identity \eqref{E:generalnormalsteinidentity2}, the second term on the right hand side in \eqref{E:lowerboundequation} is 0, and using Fubini's theorem, the right hand side can be rewritten as:
\[\mathbb{E}(h(\frac{X-\mu}{\sigma})) - \mathbb{E}(h(\frac{Z-\mu}{\sigma})) = \int_{-\infty}^\infty h(\frac{x-\mu}{\sigma})dF_X(x) - \int_{-\infty}^\infty h(\frac{x-\mu}{\sigma})dF_Z(x)\]
\[= \int_{-\infty}^\infty \int_{-\infty}^x\frac{1}{\sigma}h'(\frac{y-\mu}{\sigma})dy dF_X(x) - \int_{-\infty}^\infty \int_{-\infty}^x\frac{1}{\sigma}h'(\frac{y-\mu}{\sigma})dy dF_Z(x)\]
\[= \int_{-\infty}^\infty \frac{1}{\sigma}h'(\frac{y-\mu}{\sigma})\mathbb{P}(X>y)dy  - \int_{-\infty}^\infty\frac{1}{\sigma}h'(\frac{y-\mu}{\sigma})\mathbb{P}(Z>y)dy,\]
implying that
\[\bigl|\mathbb{E}(h(\frac{X-\mu}{\sigma})) - \mathbb{E}(h(\frac{Z-\mu}{\sigma}))\bigr| \leq \int_{-\infty}^\infty \frac{1}{\sigma}\bigl|h'(\frac{y-\mu}{\sigma})\bigr|\,\bigl|\mathbb{P}(X>y)-\mathbb{P}(Z>y)\bigr|dy\]
\[\leq d_K\Bigl(\mathscr{L}\bigl(\frac{X-\mu}{\sigma}\bigr),\textnormal{N}(0,1))\Bigr)\int_{-\infty}^\infty \frac{1}{\sigma}\bigl|h'(\frac{y-\mu}{\sigma})\bigr|dy\]
\[= d_K\Bigl(\mathscr{L}\bigl(\frac{X-\mu}{\sigma}\bigr),\textnormal{N}(0,1))\Bigr)\int_{-\infty}^\infty|h'(x)|dx.\]
We therefore get the following lower bound for the Kolmogorov distance:
\begin{equation} \label{E:Kolmogorovfirstlowerbound}
d_K\Bigl(\mathscr{L}\bigl(\frac{X-\mu}{\sigma}\bigr),\textnormal{N}(0,1))\Bigr) \geq \frac{\bigl|\mathbb{E}(\sigma^2g'(X) - (X-\mu)g(X))\bigr|}{\int_{-\infty}^\infty|h'(x)|dx}.
\end{equation}
From the definition and \eqref{E:lowerboundequation}, we get a very similar lower bound for the Wasserstein distance:
\[d_W\Bigl(\mathscr{L}\bigl(\frac{X-\mu}{\sigma}\bigr),\textnormal{N}(0,1))\Bigr) \geq \frac{\bigl|\mathbb{E}(\sigma^2g'(X) - (X-\mu)g(X))\bigr|}{\max_{x\in\mathbb{R}}|2x^3-5x|e^{-x^2/2}}.\]
We next observe that $g'(y) = \bigl(1-\frac{(y-\mu)^2}{\sigma^2}\bigr)\exp(-\frac{(y-\mu)^2}{2\sigma^2})$, and
\[g''(y) = \bigl(\frac{(y-\mu)^3}{\sigma^4}-\frac{3(y-\mu)}{\sigma^2}\bigr)\exp(-\frac{(y-\mu)^2}{2\sigma^2}) \qquad\forall y\in\mathbb{R},\]
which in turn implies: $g(\sigma x +\mu) = \sigma xe^{-x^2/2}$, $g'(\sigma x +\mu) = (1-x^2)e^{-x^2/2}$, and $g''(\sigma x +\mu) = \sigma^{-1}(x^3-3x)e^{-x^2/2}$, for each $x\in\mathbb{R}$. From this we get: $h(x)= \bigl(\sigma^2(1-x^2) - \sigma^2 x^2\bigr)e^{-x^2/2} = \sigma^2(1-2x^2)e^{-x^2/2}$, and
\[h'(x) = \sigma^2\bigl(-4x -x +2x^3\bigr)e^{-x^2/2} = \sigma^2(2x^3-5x)e^{-x^2/2} \qquad\forall x\in\mathbb{R}.\]
It remains to find a lower bound for the numerator in \eqref{E:Kolmogorovfirstlowerbound}. Using \eqref{E:expectedconditionalsteinidentity2}, we first write:
\[\mathbb{E}\bigl(\sigma^2 g'(X) - (X-\mu) g(X)\bigr) = \mathbb{E}\bigl((\sigma^2-\mathbb{V}(X|\mathscr{G})) g'(X) + (\mu-\mathbb{E}(X|\mathscr{G})) g(X)\bigr)\]
\[= \mathbb{E}\bigl((\sigma^2-\mathbb{V}(X|\mathscr{G}))\mathbb{E}(g'(X)|\mathscr{G}) + (\mu-\mathbb{E}(X|\mathscr{G}))\mathbb{E}(g(X)|\mathscr{G})\bigr).\]
After some straightforward computations, we get:
\[\mathbb{E}(g(X)|\mathscr{G}) = \int_{-\infty}^\infty (y-\mu)\exp\bigl(-\frac{(y-\mu)^2}{2\sigma^2}\bigr)\frac{1}{\sqrt{2\pi\mathbb{V}(X|\mathscr{G})}}\exp\bigl(-\frac{(y-\mathbb{E}(X|\mathscr{G}))^2}{2\mathbb{V}(X|\mathscr{G})}\bigr)dy\]
\begin{equation} \label{E:lowerboundfirstterm}
= \Bigl(\frac{\sigma^2}{\sigma^2 +\mathbb{V}(X|\mathscr{G})}\Bigr)^{3/2}\exp\bigl(-\frac{(\mu-\mathbb{E}(X|\mathscr{G}))^2}{2(\sigma^2+\mathbb{V}(X|\mathscr{G}))}\bigr)(\mathbb{E}(X|\mathscr{G})-\mu),
\end{equation}
and multiplying \eqref{E:lowerboundfirstterm} by $\mu-\mathbb{E}(X|\mathscr{G})$, we obtain:
\[(\mu-\mathbb{E}(X|\mathscr{G}))\mathbb{E}(g(X)|\mathscr{G})\]
\[= -\Bigl(\frac{\sigma^2}{\sigma^2 +\mathbb{V}(X|\mathscr{G})}\Bigr)^{3/2}\exp\bigl(-\frac{(\mu-\mathbb{E}(X|\mathscr{G}))^2}{2(\sigma^2+\mathbb{V}(X|\mathscr{G}))}\bigr)(\mu-\mathbb{E}(X|\mathscr{G}))^2,\]
an expression which is \emph{nonpositive}. Furthermore, by convexity,
\begin{equation} \label{E:convexfactorinlowerbound}
\Bigl(\frac{\sigma^2}{\sigma^2 + x}\Bigr)^{3/2} \geq \frac{1}{2^{3/2}}\bigl(1 - \frac{3}{4\sigma^2}(x-\sigma^2)\bigr) \qquad\forall x>-\sigma^2,
\end{equation}
where the right hand side is the tangent line at $x=\sigma^2$. It follows that if $\mathbb{V}(X|\mathscr{G})\geq \sigma^2$, then (compare the proof of Lemma 3.2.1 in \cite{BarbourHolstJanson1992}):
\[1 \geq \frac{1}{2^{3/2}} \geq \Bigl(\frac{\sigma^2}{\sigma^2 +\mathbb{V}(X|\mathscr{G})}\Bigr)^{3/2}\exp\bigl(-\frac{(\mu-\mathbb{E}(X|\mathscr{G}))^2}{2(\sigma^2+\mathbb{V}(X|\mathscr{G}))}\bigr)\]
\[\geq \frac{1}{2^{3/2}}\bigl(1 - \frac{3(\mathbb{V}(X|\mathscr{G})-\sigma^2)_+}{4\sigma^2} - \frac{(\mu-\mathbb{E}(X|\mathscr{G}))^2}{4\sigma^2}\bigr),\]
while if $\mathbb{V}(X|\mathscr{G})\leq \sigma^2$, then:
\[1 \geq \Bigl(\frac{\sigma^2}{\sigma^2 +\mathbb{V}(X|\mathscr{G})}\Bigr)^{3/2}\exp\bigl(-\frac{(\mu-\mathbb{E}(X|\mathscr{G}))^2}{2(\sigma^2+\mathbb{V}(X|\mathscr{G}))}\bigr)\]
\[\geq \frac{1}{2^{3/2}}\Bigl(1 + \frac{3(\sigma^2-\mathbb{V}(X|\mathscr{G}))_+}{4\sigma^2} - \frac{(\mu-\mathbb{E}(X|\mathscr{G}))^2}{2\sigma^2}\]
\[- \frac{3(\sigma^2-\mathbb{V}(X|\mathscr{G}))_+(\mu-\mathbb{E}(X|\mathscr{G}))^2}{8\sigma^4}\Bigr) \geq \frac{1}{2^{3/2}}\Bigl(1 - \frac{7(\mu-\mathbb{E}(X|\mathscr{G}))^2}{8\sigma^2}\Bigr).\]
Multiplying by $\mu-\mathbb{E}(X|\mathscr{G})$ and taking expectations in the last two sets of inequalities, we get:
\begin{equation} \label{E:nonpositiveterminlowerbound}
\begin{split}
\mathbb{E}\bigl((\mu-\mathbb{E}&(X|\mathscr{G}))^2\bigr) \geq \bigl|\mathbb{E}\bigl((\mu-\mathbb{E}(X|\mathscr{G}))g(X)\bigr)\bigr|\\
\geq \frac{1}{2^{3/2}}\Bigl(\mathbb{E}&\bigl((\mu-\mathbb{E}(X|\mathscr{G}))^2\bigr) - \frac{7\mathbb{E}\bigl((\mu-\mathbb{E}(X|\mathscr{G}))^4\bigr)}{8\sigma^2}\\
-&\frac{3\mathbb{E}\bigl((\mathbb{V}(X|\mathscr{G}) - \sigma^2)_+(\mu-\mathbb{E}(X|\mathscr{G}))^2\bigr)}{4\sigma^2}\Bigr).
\end{split}
\end{equation}
From this it follows that if the asymptotic behaviour of $X$ is such that $\sigma^{-2}\mathbb{E}\bigl((\mu-\mathbb{E}(X|\mathscr{G}))^4\bigr)$ and $\sigma^{-2}\mathbb{E}\bigl((\sigma^2-\mathbb{V}(X|\mathscr{G}))_+(\mu-\mathbb{E}(X|\mathscr{G}))^2\bigr)$ converge to 0 faster than $\mathbb{E}\bigl((\mu-\mathbb{E}(X|\mathscr{G}))^2\bigr)$, it holds that $\bigl|\mathbb{E}\bigl((\mu-\mathbb{E}(X|\mathscr{G}))g(X)\bigr)\bigr| \asymp \mathbb{E}\bigl((\mu-\mathbb{E}(X|\mathscr{G}))^2\bigr)$.

Similarly, after some computations, we obtain:
\[\mathbb{E}(g'(X)|\mathscr{G}) = \Bigl(\frac{\sigma^2}{\sigma^2 +\mathbb{V}(X|\mathscr{G})}\Bigr)^{3/2}\Bigl(1 - \frac{(\mu-\mathbb{E}(X|\mathscr{G}))^2}{\sigma^2+\mathbb{V}(X|\mathscr{G})}\Bigr)\exp\bigl(-\frac{(\mu-\mathbb{E}(X|\mathscr{G}))^2}{2(\sigma^2+\mathbb{V}(X|\mathscr{G}))}\bigr),\]
and subtracting with $\mathbb{E}(g'(Z)) = \frac{1}{2^{3/2}}$ leads to
\[\mathbb{E}(g'(X)|\mathscr{G})- \frac{1}{2^{3/2}} = \Bigl(\Bigl(\frac{\sigma^2}{\sigma^2 +\mathbb{V}(X|\mathscr{G})}\Bigr)^{3/2} - \frac{1}{2^{3/2}}\Bigr)\]
\[\times\Bigl(1 - \frac{(\mu-\mathbb{E}(X|\mathscr{G}))^2}{\sigma^2+\mathbb{V}(X|\mathscr{G})}\Bigr)\exp\bigl(-\frac{(\mu-\mathbb{E}(X|\mathscr{G}))^2}{2(\sigma^2+\mathbb{V}(X|\mathscr{G}))}\bigr)\]
\[+ \frac{1}{2^{3/2}}\Bigl(\Bigl(1 - \frac{(\mu-\mathbb{E}(X|\mathscr{G}))^2}{\sigma^2+\mathbb{V}(X|\mathscr{G})}\Bigr)\exp\bigl(-\frac{(\mu-\mathbb{E}(X|\mathscr{G}))^2}{2(\sigma^2+\mathbb{V}(X|\mathscr{G}))}\bigr) - 1\Bigr).\]
Multiplying by $\sigma^2-\mathbb{V}(X|\mathscr{G})$ and using the function $\kappa$ defined in Theorem~\ref{T:Kolmogorovlowerboundmixednormal}, we get:
\begin{equation} \label{E:firstterminlowerbound}
\begin{split}
(\sigma^2&-\mathbb{V}(X|\mathscr{G}))\Bigl(\mathbb{E}(g'(X)|\mathscr{G})- \frac{1}{2^{3/2}}\Bigr)\\
= \kappa(\mathbb{V}(X|\mathscr{G}))\Bigl(&1 - \frac{(\mu-\mathbb{E}(X|\mathscr{G}))^2}{\sigma^2+\mathbb{V}(X|\mathscr{G})}\Bigr)\exp\bigl(-\frac{(\mu-\mathbb{E}(X|\mathscr{G}))^2}{2(\sigma^2+\mathbb{V}(X|\mathscr{G}))}\bigr)\\
+ \frac{1}{2^{3/2}}(\sigma^2-\mathbb{V}(X|\mathscr{G}))&\Bigl(\Bigl(1 - \frac{(\mu-\mathbb{E}(X|\mathscr{G}))^2}{\sigma^2+\mathbb{V}(X|\mathscr{G})}\Bigr)\exp\bigl(-\frac{(\mu-\mathbb{E}(X|\mathscr{G}))^2}{2(\sigma^2+\mathbb{V}(X|\mathscr{G}))}\bigr) - 1\Bigr).
\end{split}
\end{equation}
We observe that $\kappa(\sigma^2) = 0$, and that
\[\kappa'(x) = -\Bigl(\Bigl(\frac{\sigma^2}{\sigma^2 + x}\Bigr)^{3/2} - \frac{1}{2^{3/2}}\Bigr) - (\sigma^2-x)\frac{3}{2\sigma^2}\Bigl(\frac{\sigma^2}{\sigma^2 + x}\Bigr)^{5/2} \qquad\forall x> 0,\]
so $\kappa'(x)<0$ for $x\in(0,\sigma^2)$, $\kappa'(\sigma^2)=0$, and $\kappa'(x)>0$ for $x>\sigma^2$. Moreover, by \eqref{E:convexfactorinlowerbound},
\[\kappa(x) \geq \frac{1}{2^{3/2}}\frac{3}{4\sigma^2}(\sigma^2 - x)^2 \qquad\forall x\leq \sigma^2,\]
and $\kappa'(x)\to\frac{1}{2^{3/2}}$ as $x\to\infty$. Next,
\[\kappa''(x) = \frac{3}{\sigma^2}\Bigl(\frac{\sigma^2}{\sigma^2 + x}\Bigr)^{5/2} + (\sigma^2-x)\frac{15}{4\sigma^4}\Bigl(\frac{\sigma^2}{\sigma^2 + x}\Bigr)^{7/2}\]
\[= \frac{3}{\sigma^2}\Bigl(\frac{\sigma^2}{\sigma^2 + x}\Bigr)^{5/2}\Bigl(1 + \frac{5}{4}\Bigl(\frac{\sigma^2 - x}{\sigma^2 + x}\Bigr)\Bigr) = \frac{15}{2\sigma^2}\Bigl(\frac{\sigma^2}{\sigma^2 + x}\Bigr)^{5/2}\Bigl(\frac{\sigma^2}{\sigma^2 + x} - \frac{1}{10}\Bigr)\qquad\forall x> 0,\]
and
\[\kappa'''(x) = -\frac{45}{4\sigma^4}\Bigl(\frac{\sigma^2}{\sigma^2 + x}\Bigr)^{7/2} - (\sigma^2-x)\frac{105}{8\sigma^6}\Bigl(\frac{\sigma^2}{\sigma^2 + x}\Bigr)^{9/2}\]
\[= -\frac{15}{4\sigma^4}\Bigl(\frac{\sigma^2}{\sigma^2 + x}\Bigr)^{7/2}\Bigl(3 + \frac{7}{2}\Bigl(\frac{\sigma^2 - x}{\sigma^2 + x}\Bigr)\Bigr) = -\frac{105}{4\sigma^4}\Bigl(\frac{\sigma^2}{\sigma^2 + x}\Bigr)^{7/2}\Bigl(\frac{\sigma^2}{\sigma^2 + x} - \frac{1}{14}\Bigr)\qquad\forall x> 0,\]
implying that $\kappa''(x)>0$ for $x\in(0,9\sigma^2)$, $\kappa''(9\sigma^2)=0$, and $\kappa''(x)<0$ for $x>9\sigma^2$. Moreover, $\kappa''(x)$ is strictly decreasing for $x\in[0,13\sigma^2)$. This means that for $\delta>0$ small enough, for any $x_0\in(9\sigma^2-\delta,9\sigma^2)$, it holds that $\kappa''(x_0)>0$ and $\kappa'(x_0)>\frac{1}{2^{3/2}}$. It therefore holds that $\kappa(x) \geq \frac{1}{2}\kappa''(x_0)(\sigma^2-x)^2$ for $x\in[0,x_0]$, and $\kappa(x) \geq \frac{1}{2}\kappa''(x_0)(\sigma^2-x_0)^2 + \frac{1}{2^{3/2}}(x-x_0)$ for $x\geq x_0$. It also follows from the preceding that $\kappa(x) \leq \frac{1}{2}\kappa''(0)(\sigma^2-x)^2 = \frac{27}{8\sigma^2}(\sigma^2-x)^2$ for $x\geq 0$.

Using now the fact, observed in Section 3.2 in \cite{BarbourHolstJanson1992}, that $1\geq (1-2u)e^{-u} \geq 1-3u$ for all $u\geq 0$, we obtain:
\[\kappa(\mathbb{V}(X|\mathscr{G})) \geq \kappa(\mathbb{V}(X|\mathscr{G}))\Bigl(1 - \frac{(\mu-\mathbb{E}(X|\mathscr{G}))^2}{\sigma^2+\mathbb{V}(X|\mathscr{G})}\Bigr)\exp\bigl(-\frac{(\mu-\mathbb{E}(X|\mathscr{G}))^2}{2(\sigma^2+\mathbb{V}(X|\mathscr{G}))}\bigr)\]
\[\geq \kappa(\mathbb{V}(X|\mathscr{G}))\Bigl(1 - \frac{3(\mu-\mathbb{E}(X|\mathscr{G}))^2}{2\sigma^2}\Bigr),\]
and, furthermore,
\[(\sigma^2-\mathbb{V}(X|\mathscr{G}))_+\Bigl|\Bigl(1 - \frac{(\mu-\mathbb{E}(X|\mathscr{G}))^2}{\sigma^2+\mathbb{V}(X|\mathscr{G})}\Bigr)\exp\bigl(-\frac{(\mu-\mathbb{E}(X|\mathscr{G}))^2}{2(\sigma^2+\mathbb{V}(X|\mathscr{G}))}\bigr) - 1\Bigr|\]
\[\leq \frac{3(\sigma^2-\mathbb{V}(X|\mathscr{G}))_+(\mu-\mathbb{E}(X|\mathscr{G}))^2}{2\sigma^2}.\]
Furthermore,
\[(\mathbb{V}(X|\mathscr{G})-\sigma^2)_+\Bigl|\Bigl(1 - \frac{(\mu-\mathbb{E}(X|\mathscr{G}))^2}{\sigma^2+\mathbb{V}(X|\mathscr{G})}\Bigr)\exp\bigl(-\frac{(\mu-\mathbb{E}(X|\mathscr{G}))^2}{2(\sigma^2+\mathbb{V}(X|\mathscr{G}))}\bigr) - 1\Bigr|\]
\[\leq \frac{3(\mathbb{V}(X|\mathscr{G})-\sigma^2)_+(\mu-\mathbb{E}(X|\mathscr{G}))^2}{2\sigma^2}.\]
Taking expectations in \eqref{E:firstterminlowerbound} and using the last three sets of inequalities, we get:
\[\mathbb{E}\bigl(\kappa(\mathbb{V}(X|\mathscr{G}))\bigr) + \frac{3\mathbb{E}\bigl((\sigma^2-\mathbb{V}(X|\mathscr{G}))_+(\mu-\mathbb{E}(X|\mathscr{G}))^2\bigr)}{2^{5/2}\sigma^2}\geq \mathbb{E}\bigl((\sigma^2-\mathbb{V}(X|\mathscr{G}))\mathbb{E}(g'(X)|\mathscr{G})\bigr)\]
\[\geq \mathbb{E}\bigl(\kappa(\mathbb{V}(X|\mathscr{G}))\bigr) - \frac{3\mathbb{E}\bigl(\kappa(\mathbb{V}(X|\mathscr{G}))(\mu-\mathbb{E}(X|\mathscr{G}))^2\bigr)}{2\sigma^2}\]
\[- \frac{3\mathbb{E}\bigl((\mathbb{V}(X|\mathscr{G})- \sigma^2)_+(\mu-\mathbb{E}(X|\mathscr{G}))^2\bigr)}{2^{5/2}\sigma^2}.\]
From this it follows that if the asymptotic behaviour of $X$ is such that $\sigma^{-2}\mathbb{E}\bigl(|\sigma^2-\mathbb{V}(X|\mathscr{G})|(\mu-\mathbb{E}(X|\mathscr{G}))^2\bigr)$ converges to 0 faster than $\mathbb{E}(\kappa(\mathbb{V}(X|\mathscr{G})))$ (note also that $\kappa(\mathbb{V}(X|\mathscr{G})) \leq |\sigma^2-\mathbb{V}(X|\mathscr{G})|)$, it holds that
\[\bigl|\mathbb{E}\bigl((\sigma^2-\mathbb{V}(X|\mathscr{G}))\mathbb{E}(g'(X)|\mathscr{G})\bigr)\bigr|\sim \mathbb{E}\bigl(\kappa(\mathbb{V}(X|\mathscr{G}))\bigr).\]
\end{proof}

\section*{Acknowledgements}
We wish to thank an anonymous referee for a number of insightful comments.
\\ \noindent
KB is supported by the Swedish Research Council's (Vetenskapsr\aa det) grant no.~$2017$-$04951$.

\end{document}